\documentclass [11pt,] {article}
\usepackage{amsmath,amssymb}
\usepackage{amsthm,amsfonts,amscd,epsfig,lineno}
\usepackage{xcolor}
\usepackage{multirow}
\usepackage{cite,epic}
\usepackage{adjustbox}
\usepackage{graphicx}
\usepackage[left=2cm]{geometry}
\usepackage{lscape}
\usepackage{amsthm}
\usepackage{makecell}
\usepackage{float}
\usepackage{booktabs}
\usepackage{amsmath, amssymb, amsthm}

\newtheorem{theorem}{Theorem}[section]
\newtheorem{lemma}[theorem]{Lemma}

\newtheorem{corollary}[theorem]{Corollary}
\newtheorem{definition}[theorem]{Definition}
\newtheorem{example}[theorem]{Example}
\newtheorem{construction}[theorem]{Construction}
\usepackage{booktabs}
\usepackage{multirow}
\usepackage{makecell}
\usepackage{array}
\usepackage{float}
\usepackage{caption}
\newcount\refno
\usepackage{cite}
\numberwithin{equation}{section}

\title { On the possible cardinalities of quantum Latin squares
}

\author {   \footnotesize    Yuyuan Zhang, Mingzhen Lv,   Haitao Cao\footnote{Corresponding author. \; E-mail address: caohaitao@njnu.edu.cn.  }\\
       \scriptsize    School of Mathematical Sciences, Ministry of Education Key Laboratory for NSLSCS,\\ \scriptsize Nanjing Normal University, Nanjing 210023, China. }
\date{}

\begin {document}
\parindent=0.5cm
\baselineskip=0.6cm
\maketitle
 \begin {abstract}
\baselineskip=0.6cm
 In this paper, we show that for any integer \(v \ge 8\) with \(v \ne 9,11,23\), a quantum Latin square of order $v$ exists for every cardinality \(c \in [v,\,v^{2}] \setminus \{\,v+1\,\}\).

\noindent {\bf Keywords:} quantum Latin square, cardinality\\
\end{abstract}

\section{Introduction}
\label{intro}
Quantum Latin squares, introduced by Musto and Vicary~\cite{Mu2} as a quantum-theoretic generalization of classical Latin squares, form a class of combinatorial objects with deep connections to various structures in quantum information theory. These include unitary error bases (UEBs)~\cite{Mu2}, mutually unbiased bases (MUBs)~\cite{Mu1}, maximally entangled states (AME)~\cite{Goyeneche}, and $k$-uniform states~\cite{Zang2}.

A quantum  Latin square  of order $v$  is a   $v\times v$  square, denoted as $\text{QLS}(v)$, whose entries are unit column vectors from $v$-dimensional  Hilbert space $\mathcal{H}_{v}$, and such that each row and column forms an orthonormal basis. A $\text{QLS}(v)$ can be obtained by replacing each entry $i\in\{ 0,1,\dots,v-1\}$ in a classical Latin square with the computational basis vector $|i \rangle \in \mathcal{H}_{v}$, where $|i \rangle$ is a unit column vector with its $i$ component equal to $1$.  A $\text{QLS}(v)$ is called  \emph{classical}  if all entries  are constrained to the computational basis \( \{|0\rangle, |1\rangle, \dots, |v-1\rangle\} \).

In quantum theory, two unit vectors \( |\phi\rangle, |\psi\rangle \in \mathcal{H}_v \) represent the {\it same quantum state} ($|\phi\rangle= |\psi\rangle$, regarded as {\it identical})  if there exists a real number \( \theta \in[0,2\pi)\)   such that $$|\phi\rangle = e^{\mathrm{i}\theta} |\psi\rangle,$$
where $\mathrm{i}$ is the imaginary unit. Otherwise, they are considered   {\it distinct} ($|\phi\rangle\not= |\psi\rangle$). Paczos et al. \cite{Paczos} introduced  the cardinality $c$ of a $\text{QLS}(v)$ is the number of distinct vectors in the array. It is obvious that   \( v\leq  c \leq v^2 \).

A $\text{QLS}(v)$  is called \emph{apparently quantum} if $c=v$, or \emph{genuinely quantum} if \( v < c \leq v^2 \). Since an $\text{LS}(v)$ always exists, there is a classical $\text{QLS}(v)$ with cardinality $c=v$. Genuinely quantum Latin squares of order $2$ or $3$ cannot exist \cite{Paczos}.  For the maximal case, Nechita and Pillet \cite{Nechita} introduced the concept of quantum Sudoku as a special case of quantum Latin squares, and Paczos et al. \cite{Paczos} subsequently proved the existence of quantum Sudoku of order $v^2$ with  maximum cardinality $c=v^4$.  Then, Zhang and Cao et al. \cite{Zhang1,Zhang3} almost solved the existence of a \(\text{QLS}(v)\) with $c=v^2$ with 11 possible exceptions. Finally, Zang et al. \cite{Zang6} gave a complete solution to this problem for $v\geq 4$.

For the existence of a \(\text{QLS}(v)\) with \( v \leq c \leq v^2 \),   the possible cardinalities of  quantum Latin squares of order $4$ are $4$, $6$, $8$, and $16$ \cite{Paczos,Zhang1}, while those of order $5$ are  $5$, $7$, $12$, $21$, $24$, and $25$\cite{Zhang4}.
 Recently, Zhang and Ji \cite{Zhang2,Zhang5} obtained two infinite classes of $\text{QLS}(v)$s attaining all the possible cardinalities.

\begin{theorem}  {\rm (\cite{Zhang2,Zhang5})}\label{th:1.1}
 {\rm (1)} For any integer $u$, there does not exist a $\text{QLS}(u)$ with  $c=u+1$.\\
 {\rm (2)} For any integers $u\geq2$, $m\in\{4,6\}$ and \(c \in [mu,\,(mu)^{2}] \setminus \{\,mu+1\,\}\), there exists a $\text{QLS}(mu)$ with cardinality $c$.
\end{theorem}

For further reading on quantum theory and quantum Latin squares, we recommend \cite{Claeys,Li,Mu3,Nielsen,Rather1,Rather2,Rather3,Reuttera,Zang1,Zang3,Zang4,Zang5,Zuo,Zyc}. In this paper, we investigate all possible cardinalities of quantum Latin squares. As the main result, we are to prove the following theorem.

\begin{theorem}\label{th:1.2}
For any integer \(v \ge 8\) with \(v \notin \{9,11,23\}\), and for every integer \(c \in [v, v^{2}] \setminus \{v+1\}\), there exists a \(\text{QLS}(v)\) with cardinality \(c\).
\end{theorem}

\section{Preliminaries}

\subsection{A Parameterized Vector}
\begin{definition}
For each positive integer $k$, define $f_k(x)=x^{\frac{4^k-1}{3}}$, $x \in \bigl(0,\tfrac{1}{2}\bigr)$. For  $n \ge 2$, define
\[
g_n(x)=\sqrt{\,1-\sum_{k=1}^{n-1} f_k^2(x)\,},
\]
and
\[
u_n(x)=\bigl(f_1(x),f_2(x),\dots,f_{n-1}(x),g_n(x)\bigr)^{\top}.
\]
\end{definition}
It is straightforward to verify that \(f_k(x)\in\mathbb{R}\) and \(g_n(x)\in\mathbb{R}\). Moreover, \(u_n(x)\) is a unit vector in \(\mathbb{R}^n\), and the functions \(\{f_k(x)\}_{k\ge 1}\) satisfies the following properties.

\begin{lemma}\label{lem:2.2}
(1) If $f_i(x)f_j(x)=f_p(x)f_q(x)$ holds for some positive integers $i,j,p,q$, then $\{i,j\}=\{p,q\}$.\\
(2) There  do not exist  positive integers \(k,i,j,p,q\) such that $f_k^2(x)=f_i(x)f_j(x)f_p(x)f_q(x)$.
\end{lemma}

\begin{lemma}\label{lem:2.3}
Let $n \ge 3$ and let $\alpha = (a_{1}+ b_{1}\mathrm{i},\, a_{2}+ b_{2}\mathrm{i},\,\dots,\, a_{n}+ b_{n}\mathrm{i})^{\top}\in\mathcal{H}_n$ be a nonzero vector, where $a_k,b_k\in\mathbb{R}$ for $k\in\{1,2,\dots,n\}$. For any $c\ge0$, the equation $|(u_n(x),\alpha)| = c$ has at most finitely many solutions in $x\in (0,\tfrac{1}{2})$.
\end{lemma}

\begin{proof}  Since $|(u_n(x),\alpha)|=|(u_n(x),\frac{a_n - b_n \mathrm{i}}{\sqrt{a_n^2 + b_n^2}}\,\alpha)|$ when $(a_n,b_n)\not=(0,0)$,  we may assume that $b_n=0$. So
\[
(u_n(x),\alpha) = \Bigl(\sum_{k=1}^{n-1} a_k f_k(x) + a_n g_n(x)\Bigr) + \Bigl(\sum_{k=1}^{n-1} b_k f_k(x)\Bigr)\mathrm{i}.
\]
Hence
\begin{equation}\label{eq:2.1}
|(u_n(x),\alpha)|^2
=\Bigl(\sum_{k=1}^{n-1} a_k f_k(x) + a_n g_n(x)\Bigr)^2 + \Bigl(\sum_{k=1}^{n-1} b_k f_k(x)\Bigr)^2.
\end{equation}
Assume that the equation $|(u_n(x),\alpha)|^2 = c^2$ admits infinitely many solutions in  $x \in (0, \tfrac{1}{2})$.

By Equation~\eqref{eq:2.1}, we obtain
\begin{equation}\label{eq:2.2}
\Bigl(\sum_{k=1}^{n-1} a_k f_k(x)\Bigr)^2
+
\Bigl(\sum_{k=1}^{n-1} b_k f_k(x)\Bigr)^2
+
a_n^2 g_n^2(x)
- c^2
=
-2 a_n g_n(x)\sum_{k=1}^{n-1} a_k f_k(x).
\end{equation}
Squaring both sides of Equation~(\ref{eq:2.2}) and using \(g_n^2(x) =  1 - \sum_{k=1}^{n-1} f_k^2(x)  \), we get
\begin{equation}\label{eq:2.3}
\begin{aligned}
&\Biggl[
\Bigl(\sum_{k=1}^{n-1} a_k f_k(x)\Bigr)^2
+ \Bigl(\sum_{k=1}^{n-1} b_k f_k(x)\Bigr)^2
+ a_n^2 \Bigl(1 - \sum_{k=1}^{n-1} f_k^2(x)\Bigr) - c^2
\Biggr]^2 \\
&\qquad
- 4 a_n^2 \Bigl(1 - \sum_{k=1}^{n-1} f_k^2(x)\Bigr)
\Bigl(\sum_{k=1}^{n-1} a_k f_k(x)\Bigr)^2
= 0,
\end{aligned}
\end{equation}
which has infinitely many solutions for $x \in (0, \tfrac{1}{2})$.

After substituting \(f_k(x)=x^{\frac{4^k-1}{3}}\), the left-hand side of Equation~\eqref{eq:2.3} reduces to a finite linear combination of distinct powers of \(x\). Since a nonzero polynomial has only finitely many zeros on any interval, the existence of infinitely many solutions \(x\in(0,\tfrac12)\) implies that the resulting polynomial vanishes identically. Comparing the constant term in Equation~\eqref{eq:2.3}, we obtain \(a_n^2=c^2\). A further simplification of Equation~(\ref{eq:2.3}) yields
\begin{equation}\label{eq:2.4}
\begin{aligned}
&\Biggl[
\Bigl(\sum_{k=1}^{n-1} a_k f_k(x)\Bigr)^2
+ \Bigl(\sum_{k=1}^{n-1} b_k f_k(x)\Bigr)^2
- a_n^2 \sum_{k=1}^{n-1} f_k^2(x)
\Biggr]^2 \\
&\qquad
- 4 a_n^2 \Bigl(1 - \sum_{k=1}^{n-1} f_k^2(x)\Bigr)
\Bigl(\sum_{k=1}^{n-1} a_k f_k(x)\Bigr)^2
= 0.
\end{aligned}
\end{equation}

For any \(k \in \{1,2,\ldots,n-1\}\), by Lemma~\ref{lem:2.2}, we consider the coefficients of \(f_k^2(x)\) in Equation~\eqref{eq:2.4}, and obtain
\begin{equation}\label{eq:2.5}
 4 a_n^{2} a_k^{2} = 0.
\end{equation}
Then Equation~\eqref{eq:2.4} simplifies to
\begin{equation}\label{eq:2.6}
\Bigl(\sum_{k=1}^{n-1} a_k f_k(x)\Bigr)^2
+ \Bigl(\sum_{k=1}^{n-1} b_k f_k(x)\Bigr)^2
- a_n^2    \sum_{k=1}^{n-1} f_k^2(x)= 0.
\end{equation}
For \(n \ge 3\), by comparing the coefficients of the terms \(f_1(x)f_2(x)\) and \(f_k^2(x)\) in Equation~\eqref{eq:2.6}, we obtain
\begin{equation}\label{eq:2.7}
\begin{cases}
2(a_1a_2+b_1b_2)=0,\\[4pt]
a_k^{\,2}+b_k^{\,2}-a_n^{\,2}=0, \qquad k  \in \{1,2,\ldots,n-1\}.
\end{cases}
\end{equation}

Considering Equation~\eqref{eq:2.5},  it follows that either $a_n=0$ or $a_k=0$. If $a_n=0$, then the second equation in~\eqref{eq:2.7} implies $a_k=b_k=0$. If $a_k=0$, then the two equations in~\eqref{eq:2.7} reduce to $b_1b_2=0$ and $b_k^{\,2}-a_n^{\,2}=0$.
Then we conclude that $a_n=0$, and  $b_k=0$ for  $k \in \{1,2,\ldots,n-1\}$. Hence,  $\alpha$ is a zero vector, which leads to a contradiction. Therefore,  the equation $|(u_n(x),\alpha)| = c$ admits at most finitely many solutions in $x\in (0,\tfrac{1}{2})$.
\end{proof}

\subsection{Families of Quantum Latin Squares}

From now on, when we say that $A=(|a_{ij}\rangle)$ is a $\text{QLS}(n)$, then it implies that
$|a_{i,j}\rangle\in\operatorname{span}\{|0\rangle,|1\rangle,\ldots,|n-1\rangle\}$ for all $i,j\in\{0,1,\ldots,n-1\}$. For convenience, we identify \(A\) with the set of its distinct elements, and denote its cardinality by \(|A|\).

\begin{lemma}\label{lem:2.20}
Let \(A=(|a_{i,j}\rangle)\) be a \(\text{QLS}(n)\), and let \(U\) be an \(n\times n\) unitary matrix. Define \[ UA=\bigl(U|a_{i,j}\rangle\bigr). \] Then \(UA\) is also a \(\text{QLS}(n)\) and  $|UA|=|A|$. \end{lemma}

\begin{lemma}\label{lem:2.4}
Let $|a\rangle = (a_0, a_1, \ldots, a_{n-1})^{\top}, |b\rangle = (b_0, b_1, \ldots, b_{n-1})^{\top} \in \mathcal{H}_{n}$ be unit vectors. If there exists $i \in \{0,1,\ldots,n-1\}$ such that $|a_i| \neq |b_i|$, then $|a\rangle\not=|b\rangle$.
\end{lemma}

\begin{proof} Assume that   there exists  a  \( \theta \in[0,2\pi)\)  such that $|a\rangle = e^{\mathrm{i}\theta}\,|b\rangle$. Then for any $i \in \{0,1,\ldots,n-1\}$, we have $|a_i| = |e^{\mathrm{i}\theta} b_i|= |b_i|$, which is a contradiction.
\end{proof}

\begin{lemma}\label{lem:2.5}
Let \(A\) be a \(\text{QLS}(n)\) whose first row is the computational basis \(\{|0\rangle,|1\rangle,\ldots,|n-1\rangle\}\). Then there exist two distinct integers \(i,j\in\{0,1,\ldots,n-1\}\) such that  no entry of \(A\) is of the form \(\alpha|i\rangle+\beta|j\rangle\) with \(\alpha\beta\neq0\).
\end{lemma}

\begin{proof}
Suppose that for every distinct pair \(i,j\in\{0,1,\ldots,n-1\}\), there exists an entry of \(A\) of the form \(\alpha|i\rangle+\beta|j\rangle\) with \(\alpha\beta\neq0\).  Then \(A\) contains at least \(\frac{n(n-1)}{2}\) such entries.

If \(n\) is odd, then by the pigeonhole principle some row other than the first row contains at least  \(\frac{n+1}{2}\) such entries. Since the entries in a row  are mutually orthogonal,  their supports on the computational basis must be pairwise disjoint. Each of these entries is supported on  two   basis vectors, so this row would require at least \(n+1\) distinct basis vectors, which is impossible.

If \(n\) is even, then some column contains at least \(\frac{n}{2}\) such entries. Assume that this is the \(i\)-th column. Since the entries in a column are mutually orthogonal, their supports are pairwise disjoint.  Hence these entries require at least \(n\) distinct   basis vectors. However, all their supports are contained in \(\{|0\rangle,|1\rangle,\ldots,|n-1\rangle\}\setminus\{|i\rangle\}\), which contains only \(n-1\) basis vectors, a contradiction.
\end{proof}

\begin{lemma}\label{lem:2.6}
Suppose  that  $n\ge 3$. Let   \(A=(|a_{i,j}\rangle)\) be a \(\text{QLS}(n)\) with \(|A|=c\), and let \(Y\) be a finite set of unit vectors in \(\mathcal{H}_n\). Then there exists a unitary matrix \(U\) of order \(n\) such that \(|UA|=c\) and \(UA\cap Y=\emptyset\).
\end{lemma}

\begin{proof}
For \(|a_{i,j}\rangle\in A\) and \(|y\rangle= (y_0, y_1, \ldots, y_{n-1})^{\top}\in Y\), define
\[
I(|a_{i,j}\rangle,|y\rangle)= \Bigl\{x\in(0, \tfrac{1}{2}): \bigl|(u_n(x),|a_{i,j}\rangle)\bigr| =|y_0|\Bigr\},
\]
\[
I=\bigcup_{\substack{|a_{i,j}\rangle\in A\\ |y\rangle\in Y}}I(|a_{i,j}\rangle,|y\rangle).
\]

By Lemma~\ref{lem:2.3} we know that \(I(|a_{i,j}\rangle,|y\rangle)\) is a finite set. So we can choose $x_0\in(0,\tfrac{1}{2})\setminus I$, and extend $u_n(x_0)$ to an orthonormal basis $\{u_n(x_0),u_1,\ldots,u_{n-1}\}$ of $\mathcal H_n$. Define
\[
U=(u_n(x_0),u_1,\dots,u_{n-1})^{\dagger},
\]
where \({}^\dagger\) denotes the conjugate transpose. Then $U$ is a unitary matrix  of order $n$. By Lemma~\ref{lem:2.20},  \(UA\) is a \(\text{QLS}(n)\) with \(|UA|=c\). Denote $U|a_{i,j}\rangle=(w_0, w_1, \ldots, w_{n-1})^{\top}$. Since \(\bigl|w_0\bigr|=\bigl|(u_n(x_0),|a_{i,j}\rangle)\bigr|\neq |y_0|\) for every \(|y\rangle\in Y\), Lemma~\ref{lem:2.4} implies that \(U|a_{i,j}\rangle \notin Y\). Therefore, \(UA\cap Y=\emptyset\).
\end{proof}

\begin{corollary}\label{cor:2.7}
Suppose  that  $n\ge 3$. Let $A$ and $B$ be two $\text{QLS}(n)\text{s}$ with  $|A|=c_1$ and $|B|=c_2$, respectively.  Then for any positive integer $k$, there exist $k$ $\text{QLS}(n)\text{s}$ $A_1,\dots,A_k$ of cardinality $c_1$ and $k$ $\text{QLS}(n)\text{s}$ $B_1,\dots,B_k$ of cardinality $c_2$, such that \(S\cap T = \emptyset\) for any distinct $S,T\in\{A_1,\dots,A_k,B_1,\dots,B_k\}$.
\end{corollary}

\begin{proof}
We construct  $A_1,\dots,A_k$ and  $B_1,\dots,B_k$  inductively. By Lemma~\ref{lem:2.6} there exists a unitary matrix $U$ of order $n$ such that \(UA\cap B=\emptyset\). Let $A_1=UA$ and $B_1=B$. Then \(A_1\cap B_1=\emptyset\).

Suppose that $A_1,\dots,A_{k-1}$ and $B_1,\dots,B_{k-1}$ have been constructed such that any two distinct  \(\text{QLS}(n)\text{s}\) of them are disjoint.
Let
\[
Y_k = \Bigl(\bigcup_{i=1}^{k-1} A_i\Bigr) \cup \Bigl(\bigcup_{j=1}^{k-1} B_j\Bigr).
\]
By Lemma~\ref{lem:2.6}  with \(A=A_{k-1}\) and \(Y=Y_k\), there exists a \(\text{QLS}(n)\) \(A_k\) with \(|A_k|=c_1\) such that \(A_k \cap Y_k = \emptyset\). Next, define
\[
X_k= \Bigl(\bigcup_{i=1}^{k} A_i\Bigr) \cup \Bigl(\bigcup_{j=1}^{k-1} B_j\Bigr).
\]
Applying Lemma~\ref{lem:2.6} with $A=B_{k-1}$ and  $Y=X_k$, we obtain a \(\text{QLS}(n)\) \(B_k\) with  \(|B_k|=c_2\) such that \(B_k \cap X_k = \emptyset\).
\end{proof}

\begin{definition}
Let \(h \ge 1\). For  $|a\rangle=(a_0,\ldots,a_{n-1},\ldots,a_{n+h-1})^{\top}\in\mathcal H_{n+h}$, define its projection onto the first $n$ components by $|a\rangle_{[n]}=(a_0,\ldots,a_{n-1})^{\top}\in\mathcal H_{n}$.
\end{definition}

If  \(|a\rangle_{[n]}=\mathbf{0}_n\),  where \(\mathbf{0}_n\) is the zero vector in \(\mathcal{H}_n\), then \(|a\rangle \in \operatorname{span}\{|n\rangle,\ldots,|n+h-1\rangle\}\).

\begin{lemma}\label{lem:2.9}
Suppose that \(n\ge 3\) and \(h\ge 1\). Let $A=(|a_{i,j}\rangle)$ be a $\text{QLS}(n+h)$ with $|A|=c$. Define $A^h=A \cap  \operatorname{span}\{|n\rangle,\ldots,|n+h-1\rangle\}$, and let $Y$ be a finite set of unit vectors in $\mathcal H_{n+h}$ such that $A^h\subseteq Y$. Then there exists a unitary matrix $U$ of order $n+h$ such that  $|UA|=c$ and \(UA\cap Y =A^h\).
\end{lemma}

\begin{proof}
For \(|a_{i,j}\rangle\notin A^h\) and \(|y\rangle= (y_0, y_1, \ldots, y_{n+h-1})^{\top}\in Y\), define
\[
I(|a_{i,j}\rangle,|y\rangle)= \left\{x\in(0,\tfrac12):\bigl|(u_n(x),|a_{i,j}\rangle_{[n]})\bigr|=|y_0|\right\},
\]
\[
I=
\bigcup_{\substack{|a_{i,j}\rangle\notin A^h\\ |y\rangle\in Y}}
I(|a_{i,j}\rangle,|y\rangle).
\]

Since \(|a_{i,j}\rangle \notin A^h\), it follows that \(|a_{i,j}\rangle_{[n]} \neq \mathbf{0}_n\). By Lemma~\ref{lem:2.3}, \(I(|a_{i,j}\rangle,|y\rangle)\) is a finite set. Then we can choose \(x_0\in(0,\tfrac12)\setminus I\), and extend \(u_n(x_0)\) to an orthonormal basis \(\{u_n(x_0),u_1,\ldots,u_{n-1}\}\) of \(\mathcal H_n\). Define
\[
U=
\begin{pmatrix}
\bigl(u_n(x_0),u_1,\ldots,u_{n-1}\bigr)^\dagger & 0 \\[1mm]
0  & I_h
\end{pmatrix}.
\]
where \(I_h\) is the identity matrix of order \(h\).  Then $U$ is a unitary matrix  of order $n+h$.   By Lemma~\ref{lem:2.20}, \(UA\) is a \(\text{QLS}(n+h)\) with \(|UA|=c\). Denote $U|a_{i,j}\rangle=(w_0, w_1, \ldots, w_{n+h-1})^{\top}$. Since
\(|w_0|=\bigl|\bigl(u_n(x_0),|a_{i,j}\rangle_{[n]}\bigr)\bigr|\neq|y^{(0)}|\) for every $|y\rangle\in Y$. Thus, for any $|a_{i,j}\rangle\notin  A^h$, we have $U|a_{i,j}\rangle\notin Y$. This implies that $UA\cap Y\subseteq U A^h$.

On the other hand, for any \(|a_{i,j}\rangle\in A^h\), we have \(U|a_{i,j}\rangle=|a_{i,j}\rangle \). Thus $U A^h =A^h$. Then $UA\cap Y\subseteq A^h$. Since \(A^h\subseteq Y\), it follows that $A^h\subseteq UA\cap Y$.  Then \(UA\cap Y=A^h\).
\end{proof}

\begin{lemma}\label{lem:2.10}
Suppose that \(n\ge 3\) and \(h\ge 1\).  Let $A$ and $B$ be two $\text{QLS}(n+h)\text{s}$ with $|A|=c_1$, $|B|=c_2$, and $A^{h} \subseteq B$. Then for any positive integer $k$, there exist $k$ $\text{QLS}(n+h)\text{s}$ $A_1,\dots,A_k$ of cardinality $c_1$ and $k$ $\text{QLS}(n+h)\text{s}$ $B_1,\dots,B_k$ of cardinality $c_2$,  such that \(S\cap T = A^{h}\) for any distinct $S,T\in\{A_1,\dots,A_k,B_1,\dots,B_k\}$.
\end{lemma}

\begin{proof}
By Lemma~\ref{lem:2.9}, there exists a unitary matrix $U$ of order $n+h$ such that \(UA\cap B=A^{h}\). Let $A_1=UA$ and $B_1=B$. Then \(A_1\cap B_1=A^{h}\). Suppose that $A_1,\dots,A_{k-1}$ and $B_1,\dots,B_{k-1}$ have been constructed such that  \(S\cap T=A^{h}\) for any distinct \(S,T\in \{A_1,\dots,A_{k-1},B_1,\dots,B_{k-1}\}\). Let
\[
Y_k = \Bigl(\bigcup_{i=1}^{k-1} A_i\Bigr) \cup \Bigl(\bigcup_{j=1}^{k-1} B_j\Bigr).
\]
Since  $A^h_{k-1} \subseteq Y_k$. Apply Lemma~\ref{lem:2.9} with $A=A_{k-1}$ and $Y=Y_k$, we   obtain a \(\text{QLS}(n)\) \(A_k\) with  \(|A_k|=c_1\) such that \(A_k \cap Y_k = A^{h}\). Next let
\[
X_k= \Bigl(\bigcup_{i=1}^{k} A_i\Bigr) \cup \Bigl(\bigcup_{j=1}^{k-1} B_j\Bigr).
\]
Since  $B^h_{k-1} \subseteq X_k$.  Applying Lemma~\ref{lem:2.9} with $A=B_{k-1}$ and $Y=X_k$, we obtain a \(\text{QLS}(n)\) \(B_k\) with  \(|B_k|=c_2\) such that \(B_k \cap X_k = A^{h}\).
\end{proof}

\begin{corollary}\label{cor:2.11}
Suppose that $n\ge3$. Let $A$ and $B$ be two $\text{QLS}(n+1)\text{s}$  with $|A|=c_1$ and $|B|=c_2$, respectively. Then for any positive integer $k$,  there exist $k$ $\text{QLS}(n+1)\text{s}$ $A_1,\dots,A_k$ of cardinality $c_1$ and $k$ $\text{QLS}(n+1)\text{s}$ $B_1,\dots,B_k$ of cardinality $c_2$,   such that \(S\cap T= \{|n\rangle \}\) for any distinct $S,T\in\{A_1,\dots,A_k,B_1,\dots,B_k\}$.
\end{corollary}

\begin{proof}
After suitable changes of basis, we may assume that both \(A\) and \(B\) contain \(|n\rangle\). Hence \(A^{1} =\{|n\rangle\}\) and $A^{1} \subseteq B$. The conclusion follows from Lemma~\ref{lem:2.10}  with $h=1$.
\end{proof}

\begin{corollary}\label{cor:2.12}
Suppose that $n\ge3$. Let $A$ and $B$ be two $\text{QLS}(n+2)\text{s}$  with $|A|=c_1$ and $|B|=c_2$, respectively. Then for any positive integer $k$, there exist $k$ $\text{QLS}(n+2)\text{s}$ $A_1,\dots,A_k$ of cardinality $c_1$ and $k$ $\text{QLS}(n+2)\text{s}$ $B_1,\dots,B_k$ of cardinality $c_2$,    such that \(S\cap T= \{|n\rangle,|n+1\rangle \}\) for any distinct $S,T\in\{A_1,\dots,A_k,B_1,\dots,B_k\}$.
\end{corollary}

\begin{proof}
After suitable changes of basis, the first row of both \(A\) and \(B\) is the computational basis  \(\{|0\rangle,\ldots, |n+1\rangle\}\). By Lemma~\ref{lem:2.5}, and relabelling the computational basis if necessary, we may assume that no entry of  \(A\) is of the form \(\alpha|n\rangle+\beta|n+1\rangle\) with \(\alpha\beta\neq0\).  Hence  \(A^{2}= \{|n\rangle,|n+1\rangle\}\) and $A^{2} \subseteq B$.
The conclusion follows from Lemma~\ref{lem:2.10}  with $h=2$.
\end{proof}

\subsection{Examples}
\begin{example}\label{ex:1}
The cardinalities \(6\), \(8\), and \(9\) are attainable for \(\text{QLS}(6)\text{s}\). The case \(c=6\) is realized by a classical \(\text{QLS}(6)\). We give explicit examples for \(c=8\) and \(c=9\).

Let \(|+\rangle=\frac{1}{\sqrt{2}}(|0\rangle+|1\rangle)\), \(|-\rangle=\frac{1}{\sqrt{2}}(|0\rangle-|1\rangle)\).  Let \(\omega=e^\frac{2\pi \mathrm{i}}{3}\), and define
\(|a_0\rangle=\frac{1}{\sqrt{3}}(|0\rangle+|2\rangle+|4\rangle)\),
\(|a_1\rangle=\frac{1}{\sqrt{3}}(|0\rangle+\omega|2\rangle+\omega^2|4\rangle)\), and
\(|a_2\rangle=\frac{1}{\sqrt{3}}(|0\rangle+\omega^2|2\rangle+\omega|4\rangle)\).  Then the following arrays \(A\) and \(B\) are \(\text{QLS}(6)\text{s}\) with cardinalities \(8\) and \(9\), respectively:
\[
{

\setlength{\arraycolsep}{3.0pt}
\begin{array}{c@{\quad}c}
A=
\begin{array}{|c|c|c|c|c|c|}
\hline
|+\rangle&|-\rangle&|2\rangle&|3\rangle&|4\rangle&|5\rangle\\
\hline
|-\rangle&|+\rangle&|3\rangle&|2\rangle&|5\rangle&|4\rangle\\
\hline
|2\rangle&|3\rangle&|4\rangle&|5\rangle&|0\rangle&|1\rangle\\
\hline
|3\rangle&|2\rangle&|5\rangle&|4\rangle&|1\rangle&|0\rangle\\
\hline
|4\rangle&|5\rangle&|0\rangle&|1\rangle&|2\rangle&|3\rangle\\
\hline
|5\rangle&|4\rangle&|1\rangle&|0\rangle&|3\rangle&|2\rangle\\
\hline
\end{array},
&
B=
\begin{array}{|c|c|c|c|c|c|}
\hline
|a_0\rangle&|1\rangle&|a_1\rangle&|3\rangle&|a_2\rangle&|5\rangle\\
\hline
|1\rangle&|0\rangle&|3\rangle&|2\rangle&|5\rangle&|4\rangle\\
\hline
|a_1\rangle&|3\rangle&|a_2\rangle&|5\rangle&|a_0\rangle&|1\rangle\\
\hline
|3\rangle&|2\rangle&|5\rangle&|4\rangle&|1\rangle&|0\rangle\\
\hline
|a_2\rangle&|5\rangle&|a_0\rangle&|1\rangle&|a_1\rangle&|3\rangle\\
\hline
|5\rangle&|4\rangle&|1\rangle&|0\rangle&|3\rangle&|2\rangle\\
\hline
\end{array}
\end{array}.
}
\]
\end{example}

\begin{example}\label{ex:2}
The cardinalities \(7\), \(9\), and \(10\)  are attainable for  \(\text{QLS}(7)\text{s}\). The case \(c=7\) is realized by a  classical \(\text{QLS}(7)\). We give explicit examples for \(c=9\) and \(c=10\).

Let \(|+\rangle=\frac{1}{\sqrt{2}}(|0\rangle+|1\rangle)\), \(|-\rangle=\frac{1}{\sqrt{2}}(|0\rangle-|1\rangle)\). Let \(\omega=e^\frac{2\pi \mathrm{i}}{3}\), and define
\(|b_0\rangle=\frac{1}{\sqrt{3}}(|0\rangle+|1\rangle+|2\rangle)\), \(|b_1\rangle=\frac{1}{\sqrt{3}}(|0\rangle+\omega|1\rangle+\omega^2|2\rangle)\), and
\(|b_2\rangle=\frac{1}{\sqrt{3}}(|0\rangle+\omega^2|1\rangle+\omega|2\rangle)\). Then the following arrays \(A\) and \(B\) are \(\text{QLS}(7)\text{s}\) with cardinalities \(9\) and \(10\), respectively:

\[
{

\setlength{\arraycolsep}{3.0pt}
\begin{array}{c@{\hspace{12pt}}c}
A=
\begin{array}{|c|c|c|c|c|c|c|}
\hline
|+\rangle&|-\rangle&|6\rangle&|4\rangle&|3\rangle&|2\rangle&|5\rangle\\
\hline
|-\rangle&|+\rangle&|2\rangle&|3\rangle&|5\rangle&|6\rangle&|4\rangle\\
\hline
|2\rangle&|5\rangle&|3\rangle&|6\rangle&|0\rangle&|4\rangle&|1\rangle\\
\hline
|5\rangle&|3\rangle&|4\rangle&|2\rangle&|6\rangle&|1\rangle&|0\rangle\\
\hline
|4\rangle&|2\rangle&|0\rangle&|5\rangle&|1\rangle&|3\rangle&|6\rangle\\
\hline
|6\rangle&|4\rangle&|5\rangle&|1\rangle&|2\rangle&|0\rangle&|3\rangle\\
\hline
|3\rangle&|6\rangle&|1\rangle&|0\rangle&|4\rangle&|5\rangle&|2\rangle\\
\hline
\end{array},
&
B=
\begin{array}{|c|c|c|c|c|c|c|}
\hline
|b_0\rangle&|b_1\rangle&|b_2\rangle&|5\rangle&|6\rangle&|3\rangle&|4\rangle\\
\hline
|b_1\rangle&|b_2\rangle&|b_0\rangle&|3\rangle&|4\rangle&|6\rangle&|5\rangle\\
\hline
|b_2\rangle&|b_0\rangle&|b_1\rangle&|4\rangle&|3\rangle&|5\rangle&|6\rangle\\
\hline
|4\rangle&|3\rangle&|5\rangle&|6\rangle&|0\rangle&|2\rangle&|1\rangle\\
\hline
|6\rangle&|5\rangle&|3\rangle&|1\rangle&|2\rangle&|4\rangle&|0\rangle\\
\hline
|5\rangle&|6\rangle&|4\rangle&|2\rangle&|1\rangle&|0\rangle&|3\rangle\\
\hline
|3\rangle&|4\rangle&|6\rangle&|0\rangle&|5\rangle&|1\rangle&|2\rangle\\
\hline
\end{array}
\end{array}.
}
\]
\end{example}

\section{Main constructions}
In this section, we extend the singular direct product construction   \cite{Heinrich,Stinson} and direct product construction  \cite{Donald} of Latin squares to the quantum setting, and investigate the possible cardinalities of quantum Latin squares. For our constructions, we need the following definitions.

\begin{definition}
A transversal in  a \( \text{QLS}(n) \) refers to a set of $n$ elements, each located in a distinct row and a distinct column, forming an orthonormal basis of $\mathcal{H}_{n}$.
\end{definition}

Two transversals in a \( \text{QLS}(n) \) are called {\it disjoint} if they have no overlapping cells. It has been shown in \cite{Colbourn} that for any positive integer \( n \ne 2, 6 \), there exists an \( \text{LS}(n) \) with \( n \) pairwise disjoint transversals. By replacing each entry $i\in\{ 0,1,\dots,n-1\}$ with the corresponding  ket vector \( |i\rangle \), we obtain a \( \text{QLS}(n) \) with \( n \) pairwise disjoint transversals.

\begin{definition}
An incomplete quantum Latin square $\text{IQLS}(n, h)$ is an $n \times n$ matrix  with an $h\times h$ hole in the lower-right corner,  whose entries satisfy the following conditions:
\begin{itemize}
\item[(i)] Each entry is a unit column vector from $n$-dimensional Hilbert space $\mathcal{H}_n$;
\item[(ii)] For $i \in \{0,1,\dots,n-h-1\}$, the elements in the $i$-th row and column form an orthonormal basis of $\mathcal{H}_{n}$;
\item[(iii)] For $j \in \{n-h,\dots,n-1\}$, the elements in the $j$-th row and column are linear combinations of the computational basis vectors $\{|0\rangle, |1\rangle, \dots, |n-h-1\rangle$\}.
\end{itemize}
\end{definition}

The cardinality \( c \) of an $\text{IQLS}(n, h)$ to be the number of distinct vectors  in the array.  Clearly, the cardinality  satisfies \( n \leq c \leq n^2-h^2 \).

\begin{definition}
Let \(Y\subseteq\mathcal H_n\) be a finite set.  Define the extension of \(Y\) by
\[
Y^{+}=\left\{\begin{pmatrix}|y\rangle\\0\end{pmatrix}:|y\rangle\in Y\right\}\subseteq\mathcal H_{n+1}.
\]
\end{definition}

Consider the following quantum state vectors:
$$
|a\rangle = \begin{pmatrix} a_0 \\ a_1 \\ \vdots \\ a_{n-1} \end{pmatrix} \in \mathcal{H}_n, \quad
|b\rangle = \begin{pmatrix} b_0 \\ b_1 \\ \vdots \\ b_{m-1} \end{pmatrix} \in \mathcal{H}_m, \quad
|c\rangle = \begin{pmatrix} c_0 \\ c_1 \\ \vdots \\ c_{m} \end{pmatrix} \in \mathcal{H}_{m+1}.
$$

\begin{definition}[Extended Tensor Product $\otimes_+$]
Let $h$ be a positive integer , we define the extended tensor product operation $\otimes_+$ as:
$$
|a\rangle \otimes_+ |b\rangle= \begin{pmatrix} |a\rangle \otimes |b\rangle \\ \mathbf{0}_h  \end{pmatrix} \in \mathcal{H}_{mn+h}.
$$
\end{definition}

\begin{definition}[Parameterized Tensor Product $\otimes_r$]
For each index $r \in \{0,1,\ldots,h-1\}$, we define the operation $\otimes_r$ as:
$$
|a\rangle \otimes_r |c\rangle= \begin{pmatrix} |a\rangle \otimes |c\rangle_{[m]} \\ c_m|r\rangle \end{pmatrix} \in \mathcal{H}_{mn+h},
$$
where  $|r\rangle \in \mathcal{H}_{h}$ is the $r$-th standard basis vector ($|r\rangle = (0,\ldots,0,1,0,\ldots,0)^\top$  with  $r$-th component equals to  $1$).
\end{definition}

\begin{construction}[Singular direct product construction]\label{con:1}
Let \(n,m,h\) be positive integers with \(n \ge h\).  Assume that:
\begin{itemize}
\item[(i)] There exists a classical \(\text{QLS}(n)\) \(A\) containing \(n\) pairwise disjoint transversals.

\item[(ii)] There exist \(n-h\) \(\text{QLS}(m)\text{s}\) \(B_{q}, q=0,\ldots,n-h-1\).

\item[(iii)] There exist \(h\) \(\text{IQLS}(m+1,1)\text{s}\) \(C_{r}, r = 0,\ldots,h-1\).

\item[(iv)] There exists a \(\text{QLS}(h)\) \(D\).
\end{itemize}
Then there exists a \(\text{QLS}(mn+h)\) \(M\).
\end{construction}

\begin{proof} Let \( A=(|a_{i,j}\rangle) \) be a classical \( \text{QLS}(n) \). If \( |a_{i,j} \rangle \) is in the \( p \)-th transversal, it is  denoted as \( |a_{p,i,j} \rangle \), where  \(|a_{p,i,j}\rangle \in \left\{ |e\rangle : e=0,1,\ldots,n-1 \right\}\). For the construction, we partition \(C_{r}\), for \( r=0,\ldots,h-1 \), into the following submatrices: an \(m\times m\) upper-left block \(C_{r,1}\), a \(1\times m\) lower-left block \(C_{r,2}\), and an \(m\times 1\) upper-right block \(C_{r,3}\). Then we construct a matrix \(M\) of order \(mn+h\).

The upper-left \(mn\times mn\) submatrix of \(M\) is partitioned into \(n^{2}\) blocks of size \(m\times m\). For \(i,j=0,\ldots,n-1\), the \((i,j)\)-th block \(M_{i,j}\) is defined by
\[
M_{i,j}=
\begin{cases}
|a_{q,i,j}\rangle \otimes_{+} B_{q}, &
\text{if } |a_{i,j}\rangle \text{ is in the \(q\)-th transversal},\\[2mm]
|a_{n-h+r,i,j}\rangle \otimes_{r} C_{r,1}, &
\text{if } |a_{i,j}\rangle \text{ is in the \((n-h+r)\)-th transversal}.
\end{cases}
\]
The lower-left \(h\times mn\) submatrix of \(M\) is partitioned into \(hn\) blocks of size \(1\times m\).
The $j$-th block in the \( (mn+r) \)-th row row is defined as
$
M_{mn+r,j}=|a_{n-h+r,i,j}\rangle \otimes_{r} C_{r,2}.
$
Similarly, the upper-right \(mn\times h\) submatrix is partitioned into \(hn\) blocks of size \(m\times 1\).
The $i$-th block in the \( (mn+r) \)-th column is defined as
$
M_{i,mn+r}=|a_{n-h+r,i,j}\rangle \otimes_{r} C_{r,3}.
$
And the  submatrix  of order $h$ in the lower-right corner of $M$ is denoted as $M_h$. The element in the $u$-th row and $v$-th column of $M_h$, for $u,v = 0,\dots,h-1$, is given by $ \begin{pmatrix} \mathbf{0}_{mn} \\ |d_{u,v}\rangle \end{pmatrix}$,
which corresponds to the $(mn + u)$-th row and $(mn + v)$-th column of $M$.

Based on the above construction, it has been proven that \(M\) is a $\text{QLS}(mn+h)$~\cite{Zhang3}.
\end{proof}

\begin{lemma}\label{lem:3.5}
Assume that the objects in Construction~\ref{con:1} satisfy:
\begin{itemize}
\item[(i)] \(\left|\bigcup_{q=0}^{n-h-1}B_q\right|=c_0\);

\item[(ii)] \(\left|\bigcup_{r=0}^{h-1}C_r\right|=c_1\), and \(|m\rangle\in C_{r}\) for all \(r = 0,\ldots,h-1\);

\item[(iii)] \(\left|\left(\bigcup_{r=0}^{h-1} C_r\right)\cap\left(\bigcup_{q=0}^{n-h-1} B_q^{+}\right)\right|=c_2\);

\item[(iv)] \(|D|=c_3\), and \(\{|r\rangle:r=0,\ldots,h-1\}\subseteq D\).
\end{itemize}
Then  \(M\)  is a \(\text{QLS}(mn+h)\)  with \(|M|=n(c_0+c_1-c_2-1)+c_3\).
\end{lemma}

\begin{proof}
In the classical \(\text{QLS}(n)\), each vector \(|e\rangle\in\{|0\rangle,|1\rangle,\ldots,|n-1\rangle\}\) appears exactly \(n\) times. By Construction~\ref{con:1}, the \(n-h\) occurrences of \(|e\rangle\) give rise to the blocks \(|e\rangle\otimes_{+}B_{q}\), \(q=0,\ldots,n-h-1\), via the tensor product \(\otimes_{+}\).  The remaining \(h\) occurrences correspond to the blocks \(|e\rangle\otimes_{r}C_{r}\) by \(\otimes_{r}\), \(r=0,\ldots,h-1\).   Define
\[
X_e
=
\left(
\bigcup_{q=0}^{n-h-1}
\{\,|e\rangle\otimes_{+}B_{q}\,\}
\right)
\cup
\left(
\bigcup_{r=0}^{h-1}
\{\,|e\rangle\otimes_{r}C_{r}\,\}
\right).
\]
Then the entries of \(M\) consist of the elements in \(\bigcup_{e=0}^{n-1} X_e\) and \(M_h= \left\{\begin{pmatrix}\mathbf{0}_{mn}\\|d_{u,v}\rangle\end{pmatrix}:|d_{u,v}\rangle\in D\right\}\).

First, we determine the cardinality of \(X_e\). Since distinct elements in \(\bigcup_{q=0}^{n-h-1} B_{q}\)  correspond to distinct elements in \(\bigcup_{q=0}^{n-h-1}\{\,|e\rangle\otimes_{+} B_{q}\,\}\), we have
\[
\left|
\bigcup_{q=0}^{n-h-1}\{\,|e\rangle\otimes_{+}B_{q}\,\}
\right|
=
\left|
\bigcup_{q=0}^{n-h-1} B_{q}
\right|
=c_0.
\]
Similarly, distinct elements in \(\bigcup_{r=0}^{h-1} C_{r}\) give  distinct elements in \(\bigcup_{r=0}^{h-1}\{\,|e\rangle\otimes_{r} C_{r}\,\}\),
except that the common element \(|m\rangle\in C_{r}\) is mapped to \(\begin{pmatrix}\mathbf 0_{mn}\\|r\rangle\end{pmatrix}\),
which is different for  different \(r=0,\ldots,h-1\). Denote $H=\left\{\begin{pmatrix}\mathbf 0_{mn}\\|r\rangle\end{pmatrix}:r=0,\ldots,h-1\right\}$. Hence
\[
\left|
\bigcup_{r=0}^{h-1}\{\,|e\rangle\otimes_{r} C_{r}\,\}
\right|
=
c_1+h-1.
\]
Moreover, since \(\left|\left(\bigcup_{r=0}^{h-1} C_{r}\right)\cap\left(\bigcup_{q=0}^{n-h-1}  B_{q}^{+}\right)\right|=c_2\), we have
\[
\left|
\left(
\bigcup_{r=0}^{h-1}\{\,|e\rangle\otimes_{r} C_{r}\,\}
\right)
\cap
\left(
\bigcup_{q=0}^{n-h-1}\{\,|e\rangle\otimes_{+} B_{q}\,\}
\right)
\right|
=c_2.
\]
Therefore, $|X_e|=c_0+c_1+h-c_2-1$.

Next, we consider the cardinality of \(\bigcup_{e=0}^{n-1} X_e\).  For \(|e\rangle\neq |e'\rangle\), we  have
\[
\left(
\bigcup_{q=0}^{n-h-1}\{\,|e\rangle\otimes_{+} B_{q}\,\}
\right)
\cap
\left(
\bigcup_{q=0}^{n-h-1}\{\,|e'\rangle\otimes_{+} B_{q}\,\}
\right)
=\emptyset,
\]
\[
\left(
\bigcup_{q=0}^{n-h-1}\{\,|e\rangle\otimes_{+} B_{q}\,\}
\right)
\cap
\left(
\bigcup_{r=0}^{h-1}\{\,|e'\rangle\otimes_{r} C_{r}\,\}
\right)
=\emptyset,
\]
and
\[
\left(
\bigcup_{r=0}^{h-1}\{\,|e\rangle\otimes_{r}C_{r}\,\}
\right)
\cap
\left(
\bigcup_{r=0}^{h-1}\{\,|e'\rangle\otimes_{r} C_{r}\,\}
\right)
=
H.
\]
Hence, \(X_e\cap X_{e'}=H\) for \(|e\rangle\neq |e'\rangle\).
Therefore, \[\left|\bigcup_{e=0}^{n-1} X_e\right|=n(c_0+c_1+h-c_2-1)-h(n-1)=n(c_0+c_1-c_2-1)+h.\]

Finally, we compute  \(\left|\left(\bigcup_{e=0}^{n-1} X_e\right)\cup M_h\right|\).  Since \(\{|r\rangle:r=0, \ldots,h-1\}\subseteq D\), we have  \(H\subseteq M_h\). Clearly,
\[
\left(
\bigcup_{q=0}^{n-h-1}\{\,|e\rangle\otimes_{+} B_{q}\,\}
\right)
\cap M_h
=
\emptyset,
\]
and
\[
\left(
\bigcup_{r=0}^{h-1}\{\,|e\rangle\otimes_{r} C_{r}\,\}
\right)
\cap M_h
=
H.
\]
Then
\[
\left(
\bigcup_{e=0}^{n-1} X_e
\right)
\cap M_h
=
H.
\]
Moreover, \(|M_h|=|D|=c_3\). Therefore, \(|M|=n(c_0+c_1-c_2-1)+c_3\).
\end{proof}

\begin{lemma}\label{lem:3.6}
Assume that the objects in Construction~\ref{con:1} satisfy:
\begin{itemize}
\item[(i)] \(\left|\bigcup_{q=0}^{n-h-1}B_q\right|=c_0\);

\item[(ii)] \(\left|\bigcup_{r=0}^{h-1}C_r\right|=c_1\), and \(|m\rangle\notin C_{r}\) for all \(r = 0,\ldots,h-1\);

\item[(iii)] \(\left|\left(\bigcup_{r=0}^{h-1} C_r\right)\cap\left(\bigcup_{q=0}^{n-h-1} B_q^{+}\right)\right|=c_2\);

\item[(iv)] \(|D|=c_3\).
\end{itemize}
Then  \(M\)  is a \(\text{QLS}(mn+h)\)  with \(|M|=n(c_0+c_1-c_2)+c_3\).
\end{lemma}

\begin{proof} The proof is similar to that of Lemma~\ref{lem:3.5}. Since \(|m\rangle\notin C_{r}\) for all \(r = 0,\ldots,h-1\),   we have \(|X_e|=c_0+c_1-c_2\) and \(X_e \cap X_{e'} =\emptyset\) for \(e \ne e'\).  Hence \(\left|\bigcup_{e=0}^{n-1}X_e\right|=n(c_0+c_1-c_2)\). Furthermore, \(\left(\bigcup_{e=0}^{n-1} X_e\right)\cap M_h=\emptyset \) and \(|M_h|=c_3\). Therefore, \(|M|=n(c_0+c_1-c_2)+c_3\).
\end{proof}

\begin{construction}[Direct product construction] \label{con:2}
Suppose that there exists a classical \(\text{QLS}(m)\) and \(m\) \(\text{QLS}(n)\text{s}\) with \(c\)  distinct elements. Then there exists a \(\text{QLS}(mn)\) with cardinality \(mc\).
\end{construction}

\begin{proof} Let \( A=(|a_{i,j} \rangle ) \) be a classical \( \text{QLS}(m) \), \(|a_{i,j}\rangle\in\{|0\rangle,\ldots,|m-1\rangle\}\).  Let \(B_{i} = (|b^{i}_{k,l}\rangle)\), \(i=0,\ldots,m-1\), be \(\text{QLS}(n)\text{s}\)   with \(\left|\bigcup_{i=0}^{m-1}B_i\right|=c\). We construct a matrix \(M\) of order \(mn\), partitioned into \({m}^{2}\) blocks of size \(n \times n\). For \(i,j=0,\ldots,m-1\), the \((i,j)\)-th block \(M_{i,j}\) is defined by
\(M_{i,j}=|a_{i,j}\rangle \otimes B_{i}\). Consider any two elements \( |a_{i,j}\rangle \otimes |b^{i}_{k,l}\rangle \) and \(|a_{i,j'}\rangle \otimes |b^{i}_{k,l'}\rangle \) in the \((in+k)\)-th row of \( M \),
\[
\begin{aligned}
\left(|a_{i,j}\rangle \otimes |b^{i}_{k,l}\rangle,|a_{i,j'}\rangle \otimes |b^{i}_{k,l'}\rangle\right)
&=\left(|a_{i,j}\rangle,|a_{i,j'}\rangle\right)\left(|b^{i}_{k,l}\rangle, |b^{i}_{k,l'}\rangle\right)\\
&=\delta_{jj'}\delta_{ll'}.
\end{aligned}
\]
Hence the row vectors of \(M\) form an orthonormal basis of \(\mathcal H_{mn}\). Similarly, the column vectors of \(M\) also form an orthonormal basis. Therefore, \(M\) is a \(\text{QLS}(mn)\).

For \(|e\rangle\in\{|0\rangle,\ldots,|m-1\rangle\}\), define \(X_e=\bigcup_{i=0}^{m-1}\{\,|e\rangle\otimes B_i\,\}\). Then \(|X_e|=c\), and \(X_e\cap X_{e'}=\emptyset\) for \(e\ne e'\). Hence \(|M|=\left|\bigcup_{e=0}^{m-1}X_e\right|=mc\).
\end{proof}

\section{Proof of Theorem~\ref{th:1.2}}

In this section, we prove Theorem~\ref{th:1.2}. We begin with several necessary definitions and preliminary lemmas.

\begin{definition}
Let \(I\) and \(J\) be two sets of integers. Define
\[
I+J=\{i+j:\ i\in I,\ j\in J\}.
\]
\end{definition}

\begin{lemma}\label{lem:4.1}
For \(u\ge 3\) and every integer \(c\in [4,16u]\setminus\{5\}\), there exist \(u\) \(\text{QLS}(4)\text{s}\) \(B_0,\ldots,B_{u-1}\) such that
\(\left|\bigcup_{q=0}^{u-1} B_q\right|=c\).
\end{lemma}

\begin{proof}
It is known that the possible cardinalities of \(\text{QLS}(4)\text{s}\) are \( \{4,6,8,16\}\). By Corollaries~\ref{cor:2.7}  and~\ref{cor:2.11},  for any \(c_0,c_1\in \{4,6,8,16\}\), there exist two \(\text{QLS}(4)\text{s}\) \(B_0,B_1\) with \(|B_0|=c_0\), \(|B_1|=c_1\), and  \(|B_0\cap B_1|\in\{0,1\}\).

Define \(d=|B_1|-|B_0\cap B_1|\), then \(|B_0\cup B_1|=|B_0|+d\).  Since \(|B_0\cap B_1|\in\{0,1\}\),   we obtain
\(d\in \{c_1,\,c_1-1:\ c_1\in\{4,6,8,16\}\}\cup\{0\}\), and the case \(B_0=B_1\) yields \(d=0\). Hence \(d\in \{0,3,4,5,6,7,8,15,16\}\). Therefore, the attainable cardinalities for two \(\text{QLS}(4)\text{s}\) are
\[
\{4,6,8,16\} +\{0,3,4,5,6,7,8,15,16\}=[4,32]\setminus\{5,17,18,25,26,27,28,29,30\}.
\]

It remains to show that the values in \(\{18,25,27,28,29,30\}\) are attainable. Let
\[
\scriptsize
B_0=
\begin{array}{|c|c|c|c|}
\hline
|0\rangle
&
|1\rangle
&
|2\rangle
&
|3\rangle
\\
\hline
\frac13|1\rangle+\frac23|2\rangle-\frac23|3\rangle
&
\frac13|0\rangle+\frac23|2\rangle+\frac23|3\rangle
&
\frac23|0\rangle-\frac23|1\rangle-\frac13|3\rangle
&
\frac23|0\rangle+\frac23|1\rangle-\frac13|2\rangle
\\
\hline
\frac{2}{15}|1\rangle-\frac{11}{15}|2\rangle-\frac23|3\rangle
&
\frac{2}{15}|0\rangle+\frac23|2\rangle-\frac{11}{15}|3\rangle
&
-\frac{11}{15}|0\rangle-\frac23|1\rangle-\frac{2}{15}|3\rangle
&
\frac23|0\rangle-\frac{11}{15}|1\rangle-\frac{2}{15}|2\rangle
\\
\hline
\frac{14}{15}|1\rangle-\frac{2}{15}|2\rangle+\frac13|3\rangle
&
\frac{14}{15}|0\rangle-\frac13|2\rangle-\frac{2}{15}|3\rangle
&
-\frac{2}{15}|0\rangle+\frac13|1\rangle-\frac{14}{15}|3\rangle
&
-\frac13|0\rangle-\frac{2}{15}|1\rangle-\frac{14}{15}|2\rangle
\\
\hline
\end{array},
\]

\[
\scriptsize
B_1=
\begin{array}{|c|c|c|c|}
\hline
|0\rangle
&
|1\rangle
&
\frac{1}{\sqrt2}(|2\rangle+|3\rangle)
&
\frac{1}{\sqrt2}(|2\rangle-|3\rangle)
\\
\hline
|1\rangle
&
\frac{1}{\sqrt2}(|2\rangle+|3\rangle)
&
\frac{1}{\sqrt2}(|2\rangle-|3\rangle)
&
|0\rangle
\\
\hline
\frac{1}{\sqrt2}(|2\rangle+|3\rangle)
&
\frac{1}{\sqrt2}(|2\rangle-|3\rangle)
&
|0\rangle
&
|1\rangle
\\
\hline
\frac{1}{\sqrt2}(|2\rangle-|3\rangle)
&
|0\rangle
&
|1\rangle
&
\frac{1}{\sqrt2}(|2\rangle+|3\rangle)
\\
\hline
\end{array}.
\]
Then  \(B_0\) and \(B_1\) are two \(\text{QLS}(4)\text{s}\) with \(|B_0|=16\), \(|B_1|=4\), and \(B_0\cap B_1=\{|0\rangle,|1\rangle\}\).  Hence \(|B_0\cup B_1|=18\). For \(t\in\{25,27,28,29,30\}\), define \(B_1=U_tB_0\),  where the unitary matrices \(U_t\) are given as follows:
\[
\footnotesize

\begin{array}{ccc}
U_{25}=
\begin{pmatrix}
-1&0&0&0\\
0&1&0&0\\
0&0&1&0\\
0&0&0&1
\end{pmatrix},
&
\quad
U_{27}=
\begin{pmatrix}
1&0&0&0\\
0&1&0&0\\
0&0&0&1\\
0&0&1&0
\end{pmatrix},
&
\quad
U_{28}=
\begin{pmatrix}
-1&0&0&0\\
0&-1&0&0\\
0&0&1&0\\
0&0&0&1
\end{pmatrix},
\\[5mm]
U_{29}=
\begin{pmatrix}
-\frac35&0&0&\frac45\\
0&1&0&0\\
0&0&1&0\\
\frac45&0&0&\frac35
\end{pmatrix},
&
\quad
U_{30}=
\begin{pmatrix}
-\frac35&-\frac45&0&0\\
\frac45&-\frac35&0&0\\
0&0&1&0\\
0&0&0&1
\end{pmatrix}.
&
\end{array}
\]
Since each \(U_t\) is unitary, \(U_t B_0\) is also a \(\text{QLS}(4)\) with \(|U_t B_0|=16\). Moreover,
\[
\begin{aligned}
|B_0\cap U_{25}B_0|&=7,  & |B_0\cup U_{25}B_0|&=25,\\
|B_0\cap U_{27}B_0|&=5,  & |B_0\cup U_{27}B_0|&=27,\\
|B_0\cap U_{28}B_0|&=4,  & |B_0\cup U_{28}B_0|&=28,\\
|B_0\cap U_{29}B_0|&=3,  & |B_0\cup U_{29}B_0|&=29,\\
|B_0\cap U_{30}B_0|&=2,  & |B_0\cup U_{30}B_0|&=30.
\end{aligned}
\]
Consequently, every integer in \([4,32]\setminus\{5,17,26\}\) is  attainable for two \(\text{QLS}(4)\text{s}\).

For three \(\text{QLS}(4)\text{s}\), either with repetition allowed or by adjoining a disjoint \(\text{QLS}(4)\) with cardinality in \(\{4,6,8,16\}\) by Lemma~\ref{lem:2.6}, we obtain
\[
\bigl([4,32]\setminus\{5,17,26\}\bigr)+\{0,4,6,8,16\}= [4,48]\setminus\{5,42\}.
\]
The remaining value \(42\) is obtained by taking \(B_1=U_{28}B_0\) and \(B_2=U_{30}B_0\), which yields \(|B_0\cup B_1\cup B_2|=42\). Hence every integer in \([4,48]\setminus\{5\}\) is attainable by three \(\text{QLS}(4)\text{s}\).

We proceed by induction on \(u\ge 3\). Assume that every integer in \([4,16u]\setminus\{5\}\) is attainable by \(u\) \(\text{QLS}(4)\text{s}\). For \(u+1\), all values in \([4,16u]\setminus\{5\}\) remain attainable by taking a repeated \(\text{QLS}(4)\). Moreover, adjoining a disjoint \(\text{QLS}(4)\) of cardinality \(16\) yields additional attainable values in \([20,16u+16]\setminus\{21\}\). Therefore, the attainable cardinalities for \(u+1\) \(\text{QLS}(4)\text{s}\) are
\[
([4,16u]\setminus\{5\})\cup([20,16u+16]\setminus\{21\})=[4,16(u+1)]\setminus\{5\}.
\]
This completes the induction.
\end{proof}

Recently, Zhang et al.~\cite{Zhang2} proved that, for any integers \(u\geq 2\) and   \(c \in [4u,16u^2]\setminus\{4u+1\}\),
there exists a \(\text{QLS}(4u)\) of cardinality \(c\). For self-contained we provide an alternative proof of this result.

\begin{lemma}\label{lem:4.2}
For any integers $u \ge 2$ and  \(c \in [4u,\,16u^2] \setminus \{\,4u+1\,\}\), there exists a $\text{QLS}(4u)$ with cardinality $c$.
\end{lemma}

\begin{proof} By Lemma~\ref{lem:4.1}, the union of \(u\) \(\text{QLS}(4)\text{s}\) can attain every integer in \([4,32]\setminus\{5,17,26\}\) for \(u=2\), and in \([4,16u]\setminus\{5\}\) for \(u\ge 3\). Applying Construction~\ref{con:2} with \(m=u\) and \(n=4\), we obtain that the \(u\)-fold sums of these sets yield \([4u,16u^2]\setminus\{4u+1\}\). Hence every \(c\in[4u,16u^2]\setminus\{4u+1\}\) is attainable for \(\text{QLS}(4u)\).
\end{proof}

\begin{lemma}\label{lem:4.20}
For \(u\ge 3\) and every integer  \(c\in [5,25u]\setminus\{5\}\), there exist \(u\) \(\text{QLS}(5)\text{s}\) \(B_0,\ldots,B_{u-1}\) such that
\(\left|\bigcup_{q=0}^{u-1} B_q\right|=c\).
\end{lemma}

\begin{proof}
It is known from \cite{Zhang4} that the possible cardinalities of \(\text{QLS}(5)\text{s}\) are \(\{5,7,12,21,24,25\}\).  By Corollaries~\ref{cor:2.7},   \ref{cor:2.11} and~\ref{cor:2.12},  for any \(c_0,c_1\in \{5,7,12,21,24,25\}\), there exist two \(\text{QLS}(5)\text{s}\) \(B_0,B_1\) with \(|B_0|=c_0\), \(|B_1|=c_1\), and  \(|B_0\cap B_1|\in\{0,1,2\}\).

Define \(d=|B_1|-|B_0\cap B_1|\). The possible values of \(d\) include \(\{c_1,c_1-1,c_1-2:c_1\in \{5,7,12,21,24,25\}\cup\{0\} \}\).  Then \( d \in \{0\}\cup [3,7]\cup [10,12] \cup[19,25] \). Hence the attainable cardinalities for two \(\text{QLS}(5)\text{s}\) include
\[
\{5,7,12,21,24,25\} +\bigl(\{0\}\cup [3,7]\cup [10,12] \cup[19,25] \bigr)=[5,50]\setminus\{6, 20,38,39\}.
\]

For three \(\text{QLS}(5)\text{s}\), by taking a repeated \(\text{QLS}(5)\), or by adjoining a disjoint \(\text{QLS}(5)\) with cardinality in \(\{5,7,12,21,24,25\}\) by Lemma~\ref{lem:2.6}, we obtain
\[
\bigl([5,50]\setminus\{6, 20,38,39\} \bigr) + \{0,5,7,12,21,24,25\}= [5,75]\setminus\{6\}.
\]
Hence every integer in \([5,75]\setminus\{6\}\) is attainable by three \(\text{QLS}(5)\text{s}\). Then the result follows for all \(u\ge 3\) by induction.
\end{proof}

\begin{lemma}\label{lem:4.3}
Let \(Y\subseteq\mathcal H_4\) be a finite set, and  \(Y^{+}\subseteq\mathcal H_{5}\). For \(u\in\{1,2,3,7\}\), there exist \(u\) \(\text{IQLS}(5,1)\text{s}\) \(C_0,\ldots,C_{u-1}\) such that \(|4\rangle\in C_r\) for all \(r=0,\ldots,u-1\), and \(\left(\bigcup_{r=0}^{u-1} C_r\right)\cap Y^{+}=\emptyset\). Moreover, the attainable cardinalities of \(\left|\bigcup_{r=0}^{u-1} C_r\right|\) are listed in Table~\ref{tab:1}.
\end{lemma}

\begin{proof}
For \(c\in \{5,7,12,21,24\}\), elements in a \(\text{QLS}(5)\) of cardinality \(c\) may repeat. Without loss of generality, assume that \(|4\rangle\) is a repeated element. By Corollary~\ref{cor:2.11}, there exist sufficiently many \(\text{QLS}(5)\text{s}\) with cardinalities in \(\{5,7,12,21,24\}\), each containing \(|4\rangle\), and any two of them intersect exactly in \(\{|4\rangle\}\). Since \(Y^{+}\) is finite and \(|4\rangle\notin Y^{+}\), each element of \(Y^{+}\) appears in at most one such \(\text{QLS}(5)\). Hence only finitely many of these \(\text{QLS}(5)\text{s}\) contain elements of \(Y^{+}\). Therefore, by selecting sufficiently many such \(\text{QLS}(5)\text{s}\), we may choose \(u\) of them whose entries avoid \(Y^{+}\).

Since \(|4\rangle\) is a repeated element, removing \(|4\rangle\) from each chosen \(\text{QLS}(5)\) does not change its cardinality. We thus obtain \(u\) \(\text{IQLS}(5,1)\text{s}\) \(C_0,\ldots,C_{u-1}\) satisfying \(C_r\cap C_{r'}=\{|4\rangle\}\) for \(r\neq r'\), and \(\left(\bigcup_{r=0}^{u-1} C_r\right)\cap Y^{+}=\emptyset\). A disjoint \(\text{IQLS}(5,1)\) of cardinality \(c\in\{5,7,12,21,24\}\) containing \(|4\rangle\) contributes \(c-1\) elements, while a repeated one contributes \(0\). Hence, the possible increments are given by \(\{0,4,6,11,20,23\}\). Starting from \(\{5,7,12,21,24\}\), the attainable cardinalities obtained by additions are listed in Table~\ref{tab:1}.
\end{proof}

\begin{table}[H]
\centering
\small

\setlength{\tabcolsep}{9pt}
\caption{Attainable cardinalities of \(\left|\bigcup_{r=0}^{u-1} C_r\right|\),    \(|4\rangle\in C_r\) }
\label{tab:1}
\begin{tabular}{c p{12.5cm}}
\toprule
\(u\) & Attainable cardinalities \\
\midrule
1 &
\(S_1=\{5,7,12,21,24\}\)
\\[2mm]
2 &
\(S_2=\{5,7,9,11,12,13,16,18,21,23,24,25,27,28,30,32,35,41,44,47\}\)
\\[2mm]
3 &
\(S_3=\left[5,70\right] \setminus \{6,8,10,14,26,37,40,42,49,54,56,57,59,60,62,63,65,66,68,69\}\)
\\[2mm]
7 &
\(S_7=\left[5,162\right] \setminus  \{6,8,10,14,146,148,149,151,152,154,155,157,158,160,161\}\)
\\
\bottomrule
\end{tabular}
\end{table}

\begin{lemma}\label{lem:4.4}
Let \(Y\subseteq\mathcal H_4\) be a finite set, and  \(Y^{+}\subseteq\mathcal H_{5}\). For \(u\in\{1,2,3,7\}\), there exist \(u\) \(\text{IQLS}(5,1)\text{s}\) \(C_0,\ldots,C_{u-1}\) such that \(|4\rangle\notin C_r\) for all \(r=0,\ldots,u-1\), and \(\left(\bigcup_{r=0}^{u-1} C_r\right)\cap Y^{+}=\emptyset\). Moreover, the attainable cardinalities of \(\left|\bigcup_{r=0}^{u-1} C_r\right|\) are listed in Table~\ref{tab:2}.
\end{lemma}

\begin{proof}
For \(c \in \{21,24,25\}\), every \(\text{QLS}(5)\) of cardinality \(c\) contains an element appearing exactly once. Without loss of generality,    assume that \(|4\rangle\) occurs exactly once. By   Corollary~\ref{cor:2.11}, there exist sufficiently many \(\text{QLS}(5)\text{s}\)  with cardinalities in \(\{ 21,24,25\}\), each containing \(|4\rangle\), and any two of them intersect exactly in \(\{|4\rangle\}\).  Hence, we may select \(u\) \(\text{QLS}(5)\text{s}\) whose entries avoid   elements of \(Y^+\).

Since \(|4\rangle\) occurs exactly once, removing \(|4\rangle\)  reduces the cardinality by one. Hence, we obtain \(u\) \(\text{IQLS}(5,1)\text{s}\)  \(C_0,\ldots,C_{u-1}\) with cardinalities in \(\{20,23,24\}\), satisfying \(|4\rangle \notin C_r\),  \(C_r \cap C_{r'}=\emptyset\) for \(r\neq r'\), and \(\left(\bigcup_{r=0}^{u-1} C_r\right)\cap Y^{+}=\emptyset\).   A disjoint \(\text{IQLS}(5,1)\) of cardinality \(c\in\{20,23,24\}\) contributes \(c\) elements, while a repeated one contributes \(0\). Hence, the possible increments are \(\{0,20,23,24\}\). Starting from \(\{20,23,24\}\), the attainable cardinalities obtained by additions are listed in Table~\ref{tab:2}.
\end{proof}

\begin{table}[H]
\centering
\small

\setlength{\tabcolsep}{9pt}
\caption{Attainable cardinalities of \(\left|\bigcup_{r=0}^{u-1} C_r\right|\),   \(|4\rangle\notin C_r\) }
\label{tab:2}
\begin{tabular}{c p{12.2cm}}
\toprule
\(u\) & Attainable cardinalities \\
\midrule
1 &
\(T_1=\{20,23,24\}\)
\\[2mm]
2 &
\(T_2=\{20, 23, 24, 40, 43, 44, 46, 47, 48\}\)
\\[2mm]
3 &
\( T_3=\{20, 23, 24, 40, 43, 44, 46, 47, 48, 60, 63, 64, 66, 67, 68, 69, 70, 71, 72\}\)
\\[2mm]

7 &
\(T_7=\begin{aligned}
&\{20,23,24,40,43,44,46,47,48,60,63,64,  80,83,84,   100,103,104,   123,124\}
  \\
& \cup [66,72] \cup[86,96]
 \cup[106,120]  \cup[126,144]\cup[146,168].
\end{aligned}\)
\\
\bottomrule
\end{tabular}
\end{table}

\begin{lemma}\label{lem:4.21}
Let \(m\ge 3\). Let \(Y\subseteq \mathcal H_m\) be a finite set, and   \(Y^{+}\subseteq \mathcal H_{m+1}\). For any positive integer  \(u\) and any  integer \(c\in \{u_0((m+1)^2-1):  u_0=1,\ldots,u  \} \), there exist \(u\) \(\text{IQLS}(m+1,1)\text{s}\) \(C_0,\ldots,C_{u-1}\) such that \(|m\rangle\notin C_r\)  for  \(r=0,\ldots,u-1\), \(\left(\bigcup_{r=0}^{u-1}C_r\right)\cap Y^{+}=\emptyset\), and \(\left|\bigcup_{r=0}^{u-1}C_r\right|=c\).
\end{lemma}

\begin{proof}
Since \(m\ge 3\), there exists a \(\text{QLS}(m+1)\) of maximal cardinality \((m+1)^2\). By Corollary~\ref{cor:2.11}, there exist sufficiently many  \(\text{QLS}(m+1)\text{s}\)  with cardinality \((m+1)^2\), each containing \(|m\rangle\), and any two of them intersect exactly in \(\{|m\rangle\}\). Since \(Y^{+}\) is finite and \(|m\rangle\notin Y^{+}\),  we may choose \(u\) \(\text{QLS}(m+1)\text{s}\) whose entries avoid  \(Y^+\).

For each chosen \(\text{QLS}(m+1)\), deleting the unique occurrence of \(|m\rangle\) yields an \(\text{IQLS}(m+1,1)\) of cardinality \((m+1)^2-1\). Moreover, since any two of the original \(\text{QLS}(m+1)\text{s}\) intersect only in \(|m\rangle\), the resulting \(\text{IQLS}(m+1,1)\text{s}\) are pairwise disjoint and avoid \(Y^{+}\).

Let \(c=u_0((m+1)^2-1)\), \(u_0=1,\ldots,u\). Choose \(u_0\) pairwise disjoint \(\text{IQLS}(m+1,1)\text{s}\) as above, and let the remaining \(u-u_0\) members be repetitions of one of them. Then the union of the resulting \(u\) \(\text{IQLS}(m+1,1)\text{s}\) has cardinality \(c\).
\end{proof}

\begin{lemma}\label{lem:4.5}
For  $u \ge 3$ and every integer \(c \in [4u+1,\,(4u+1)^{2}] \setminus \{4u+2\}\), there exists a $\text{QLS}(4u+1)$ with cardinality $c$.
\end{lemma}

\begin{proof}
For \(u\ge 4\) and \(u\ne 6\), there exists a \(\text{QLS}(u)\) \(A\) containing \(u\) pairwise disjoint transversals. Moreover, by Lemma~\ref{lem:4.1}, for any \(c_0\in [4,16(u-1)]\setminus\{5\}\), there exist \(u-1\) \(\text{QLS}(4)\text{s}\) \(B_0,\ldots,B_{u-2}\) such that \(\left|\bigcup_{q=0}^{u-2}B_q\right|=c_0\). Let \(D\) be a classical \(\text{QLS}(1)\). Applying Construction~\ref{con:1} with \(n=u\), \(m=4\), and \(h=1\), we obtain a \(\text{QLS}(4u+1)\), and it remains to determine the required \(\text{IQLS}(5,1)\) \(C_0\).

\textbf{Case 1.}
Let \(C_0\) be obtained by extending a basis contained in some \(B_q\), together with \(|4\rangle\). Thus Lemma~\ref{lem:3.5} applies with \(c_1=5\), \(c_2=4\), and \(c_3=1\), and for every integer \(c\in [4u+1,\,16u(u-1)+1]\setminus\{4u+2\}\), there exists a \(\text{QLS}(4u+1)\) with cardinality \(c\).

\textbf{Case 2.}
By Lemma~\ref{lem:4.4}, we may choose \(C_0\) such that \(|C_0|=24\), \(|4\rangle\notin C_0\), and \(C_0\cap \left(\bigcup_{q=0}^{u-2}B_q^+\right)=\emptyset\). Thus Lemma~\ref{lem:3.6} applies with \(c_1=24\), \(c_2=0\), and \(c_3=1\), and we obtain \(\text{QLS}(4u+1)\text{s}\) with cardinalities \(c\in [28u+1,\,(4u+1)^2]\setminus\{28u+2\}\).

Since \(28u + 2 \le 16u(u-1) + 1\) for \(u \ge 4\), every integer \(c \in [4u+1,(4u+1)^2]\setminus\{4u+2\}\) is attainable. For \(u\in\{3,6\}\), Table~\ref{tab:4u+1} shows that all cardinalities are attainable.
\end{proof}

\begin{table}[htbp]
\centering
{\footnotesize
\setlength{\tabcolsep}{5pt}
\caption{Cardinalities of $\text{QLS}(4u+1)$s for $u\in \{3,6\}$}
\label{tab:4u+1}
\begin{tabular}{
c c
@{\hspace{5pt}}
p{3.5cm}
@{\hspace{3pt}}
p{5.5cm}
@{\hspace{10pt}}   
c
}

\toprule
$u$ & $4u+1$ & Constructions & Input designs & Cardinalities \\
\midrule

\multirow{2}{*}{$3$} & \multirow{2}{*}{$13$}
&
\begin{tabular}[t]{@{}l@{}}
Lemma~\ref{lem:3.5},\\
$n=3,m=4,h=1$
\end{tabular}
&
\begin{tabular}[t]{@{}l@{}}
 Lemma~\ref{lem:4.1}, $c_0\in [4,32]\setminus\{5,17,26\}$;\\
$c_1=5,\, c_2=4,\, c_3=1$.
\end{tabular}
&
$[13,97]\setminus\{14\}$ \\

\cmidrule(lr){3-5}

& &
\begin{tabular}[t]{@{}l@{}}
Lemma~\ref{lem:3.6},\\
$n=3,m=4,h=1$
\end{tabular}
&
\begin{tabular}[t]{@{}l@{}}
 Lemma~\ref{lem:4.1},  $c_0\in [4,32]\setminus\{5,17,26\}$;\\
 Lemma~\ref{lem:4.4}, $c_1=24,\, c_2=0;\, c_3=1$.
\end{tabular}
&
$[85,169]\setminus\{86\}$ \\
\midrule

\multirow{2}{*}{$6$} & \multirow{2}{*}{$25$}
&
\begin{tabular}[t]{@{}l@{}}
Construction~\ref{con:2}\\
$m=5,n=5$
\end{tabular}
&
\begin{tabular}[t]{@{}l@{}}
Lemma~\ref{lem:4.20}, five  \(\text{QLS}(5)\text{s}\)  with\\ $c\in [5,125]\setminus\{6\}$.
\end{tabular}
&
$[25,625]\setminus\{26\}$ \\

\bottomrule
\end{tabular}
} 
\end{table}

\begin{lemma}\label{lem:4.6}
For   $u \ge 2$ and  every integer \(c \in [4u+2,\,(4u+2)^{2}] \setminus \{ 4u+3 \}\), there exists a $\text{QLS}(4u+2)$ with cardinality $c$.
\end{lemma}

\begin{proof}
For \(u\ge 5\) and \(u\ne 6\), there exists a \(\text{QLS}(u)\) \(A\) containing \(u\) pairwise disjoint transversals. By Lemma~\ref{lem:4.1}, for any \(c_0\in [4,\,16(u-2)]\setminus\{5\}\), there exist \(u-2\) \(\text{QLS}(4)\text{s}\) \(B_0,\ldots,B_{u-3}\) such that \(\left|\bigcup_{q=0}^{u-3}B_q\right|=c_0\). Let \(D\) be a classical \(\text{QLS}(2)\). Applying Construction~\ref{con:1} with \(n=u\), \(m=4\), and \(h=2\), it remains to choose the two required  \(\text{IQLS}(5,1)\text{s}\) \(C_0\) and \(C_1\).

\textbf{Case 1.}
Let \(C_0=C_1\) be obtained by extending a basis contained in some \(B_q\), together with \(|4\rangle\).  Thus Lemma~\ref{lem:3.5} applies with \(c_1=5\), \(c_2=4\), and \(c_3=2\), and all cardinalities \(c\in [4u+2,\,16u(u-2)+2]\setminus\{4u+3\}\) are attainable by \(\text{QLS}(4u+2)\text{s}\).

\textbf{Case 2.}
By Lemma~\ref{lem:4.3}, we may choose \(C_0\) and \(C_1\) such that \(|4\rangle\in C_0\), \(|4\rangle\in C_1\), \(\left|\bigcup_{r=0}^{1}C_r\right|=24\), and \(\left(\bigcup_{r=0}^{1}C_r\right)\cap \left(\bigcup_{q=0}^{u-3}B_q^+\right)=\emptyset\). Thus Lemma~\ref{lem:3.5} applies with \(c_1=24\), \(c_2=0\), and \(c_3=2\), and all cardinalities \(c\in [27u+2,\,16u^2-9u+2]\setminus\{27u+3\}\) are attainable.

\textbf{Case 3.}
By Lemma~\ref{lem:4.4}, we may choose \(C_0\) and \(C_1\) such that \(|4\rangle\notin \bigcup_{r=0}^{1}C_r\), \(\left|\bigcup_{r=0}^{1}C_r\right|=48\), and \(\left(\bigcup_{r=0}^{1}C_r\right)\cap \left(\bigcup_{q=0}^{u-3}B_q^+\right)=\emptyset\). Thus Lemma~\ref{lem:3.6} applies with \(c_1=48\), \(c_2=0\), and \(c_3=2\), and all cardinalities \(c\in [52u+2,\,(4u+2)^2-2]\setminus\{52u+3\}\) are attainable.

Combining the above three cases, every integer \(c\in [4u+2,\,(4u+2)^2-2]\setminus\{4u+3\}\) is attainable as the cardinality of a \(\text{QLS}(4u+2)\). The cardinalities \(\{(4u+2)^2-1,\,(4u+2)^2\}\) are handled in Case~4.

\textbf{Case 4.}
Since there exists a \(\text{QLS}(2u+1)\) of maximal cardinality \((2u+1)^2\), by Corollaries~\ref{cor:2.7}, \ref{cor:2.11}, and~\ref{cor:2.12}, there exist two \(\text{QLS}(2u+1)\text{s}\) of cardinality \((2u+1)^2\) with intersection size in \(\{0,1,2\}\). Hence the number of distinct elements in the  two \(\text{QLS}(2u+1)\text{s}\) ranges over \([2(2u+1)^2-2,\,2(2u+1)^2]\). Since \(4u+2=2(2u+1)\), Construction~\ref{con:2} with \(m=2\) and \(n=2u+1\) yields \(\text{QLS}(4u+2)\text{s}\) with cardinalities \(c\in [(4u+2)^2-4,\,(4u+2)^2]\).

Therefore, all cardinalities \(c\in [4u+2,\,(4u+2)^2]\setminus\{4u+3\}\) are attainable by \(\text{QLS}(4u+2)\text{s}\) for \(u\ge 5\), \(u\ne 6\). The cases \(u\in\{2,3,4,6\}\) follow from Table~\ref{tab:4u+2}.
\end{proof}

\begin{table}[htbp]
\centering
{\footnotesize
\setlength{\tabcolsep}{3pt}
\caption{Cardinalities of $\text{QLS}(4u+2)$s for $u \in \{2,3,4,6\}$}
\label{tab:4u+2}
\begin{tabular}{
c c
@{\hspace{5pt}}
p{3.5cm}
@{\hspace{3pt}}
p{6.0cm}
@{\hspace{38pt}}   
c
}

\toprule
$u$ & $4u+2$ & Constructions & Input designs & Cardinalities \\
\midrule

\multirow{2}{*}{$2$} & \multirow{2}{*}{$10$}
&
\begin{tabular}[t]{@{}l@{}}
Construction~\ref{con:2}\\
$m=2,n=5$
\end{tabular}
&
\begin{tabular}[t]{@{}l@{}}
Lemma~\ref{lem:4.20},  two  \(\text{QLS}(5)\text{s}\)  with\\ $c\in [5,50]\setminus\{6,20,38,39\}$.
\end{tabular}
&
$[10,100]\setminus\{11\}$ \\

\midrule

\multirow{2}{*}{$3$} & \multirow{2}{*}{$14$}
&
\begin{tabular}[t]{@{}l@{}}
Construction~\ref{con:2},\\
$m=2,n=7$
\end{tabular}
&
\begin{tabular}[t]{@{}l@{}}
Example~\ref{ex:2}, $\text{QLS}(7)\text{s}$ admit cardinalities\\ $\{7,9,10\}$; Corollaries~\ref{cor:2.7},   \ref{cor:2.11}, \ref{cor:2.12}, two \\ \(\text{QLS}(7)\text{s}\) with \(c\in[7,20]\setminus\{8,11\}\).
\end{tabular}
&
$[14,40]\setminus\{15\}$ \\
\cmidrule(lr){3-5}

& &
\begin{tabular}[t]{@{}l@{}}
Lemma~\ref{lem:3.5},\\
$n=3,m=4,h=2$
\end{tabular}
&
\begin{tabular}[t]{@{}l@{}}
$c_0\in \{4,6,8,16\}$;
Lemma~\ref{lem:4.3}, $c_1\in \{9,11,12\}$,\\$c_2=0;\, c_3=2$.
\end{tabular}
&
$ \left[38,83\right]\setminus\{ 39\} $
\\

\cmidrule(lr){3-5}

& &
\begin{tabular}[t]{@{}l@{}}
Lemma~\ref{lem:3.6},\\
$n=3,m=4,h=2$
\end{tabular}
&
\begin{tabular}[t]{@{}l@{}}
$c_0\in \{4,6,8,16\}$;
Lemma~\ref{lem:4.4}, $c_1\in T_2$,\,$ c_2=0;$\\$c_3=2$.
\end{tabular}
&
$[74,194]\setminus\{75\}$
\\
\cmidrule(lr){3-5}

& &
\begin{tabular}[t]{@{}l@{}}
Construction~\ref{con:2},\\
$m=2,n=7$
\end{tabular}
&
\begin{tabular}[t]{@{}l@{}}
\(\text{QLS}(7)\text{s}\) admit  cardinality \(49\);
Corollaries~\ref{cor:2.7},\\   \ref{cor:2.11},
 \ref{cor:2.12}, two   \(\text{QLS}(7)\text{s}\) with  $c\in [96,98]$.
\end{tabular}
&
$[192,196]$ \\
\midrule

\multirow{2}{*}{$4$} & \multirow{2}{*}{$18$}
&
\begin{tabular}[t]{@{}l@{}}
Lemma~\ref{lem:3.5},\\
$n=4,m=4,h=2$
\end{tabular}
&
\begin{tabular}[t]{@{}l@{}}
Lemma~\ref{lem:4.1}, $c_0\in [4,32]\setminus\{5,17,26\}$;
$c_1=5,$\\ $c_2=4, c_3=2$.
\end{tabular}
&
$[18,130]\setminus\{19\}$ \\

\cmidrule(lr){3-5}

& &
\begin{tabular}[t]{@{}l@{}}
Lemma~\ref{lem:3.6},\\
$n=4,m=4,h=2$
\end{tabular}
&
\begin{tabular}[t]{@{}l@{}}
Lemma~\ref{lem:4.1}, $c_0\in [4,32]\setminus\{5,17,26\}$;\\
Lemma~\ref{lem:4.4}, $c_1 \in T_2,\, c_2=0; c_3=2$.
\end{tabular}
&
$[98,322]\setminus\{99\}$ \\

\cmidrule(lr){3-5}

& &
\begin{tabular}[t]{@{}l@{}}
Construction~\ref{con:2},\\
$m=2,n=9$
\end{tabular}
&
\begin{tabular}[t]{@{}l@{}}
\(\text{QLS}(9)\text{s}\) admit  cardinality \(81\);
Corollaries~\ref{cor:2.7},\\   \ref{cor:2.11},  \ref{cor:2.12},  two   \(\text{QLS}(9)\text{s}\) with $c\in [160,162]$.
\end{tabular}
&
$[320,324]$ \\
\midrule

\multirow{2}{*}{$6$} & \multirow{2}{*}{$26$}
&
\begin{tabular}[t]{@{}l@{}}
Construction~\ref{con:2}\\
$m=2,n=13$
\end{tabular}
&
\begin{tabular}[t]{@{}l@{}}
Lemma~\ref{lem:4.5}, $\text{QLS}(13)\text{s}$ admit cardinalities\\ $[13,169] \setminus\{14\}$; Corollaries~\ref{cor:2.7}, two  \(\text{QLS}(13)\text{s}\)\\ with \(c\in [13,338]\setminus\{14\}\).
\end{tabular}
&
$[26,676]\setminus\{27\}$ \\

\bottomrule
\end{tabular}
} 
\end{table}

\begin{lemma}\label{lem:4.7}
For  $u \ge 2, u\neq 4$ and  every integer \(c \in [4u+7,\,(4u+7)^{2}] \setminus \{ 4u+8 \}\), there exists a $\text{QLS}(4u+7)$ with cardinality $c$.
\end{lemma}
\begin{proof}
For \(u\ge 10\), there exists a \(\text{QLS}(u)\) containing \(u\) pairwise disjoint transversals. By Lemma~\ref{lem:4.1}, for any \(c_0\in [4,\,16(u-7)]\setminus\{5\}\), there exist \(u-7\) \(\text{QLS}(4)\text{s}\) \(B_0,\ldots,B_{u-8}\) such that \(\left|\bigcup_{q=0}^{u-8}B_q\right|=c_0\). Applying Construction~\ref{con:1} with \(n=u\), \(m=4\), and \(h=7\),  it remains to choose the required seven \(\text{IQLS}(5,1)\text{s}\) \(C_0,\ldots,C_6\) and a \(\text{QLS}(7)\) \(D\).

\textbf{Case 1.}
Let \(C_0,\ldots,C_6\) be identical \(\text{IQLS}(5,1)\text{s}\) consisting of the extensions of a basis contained in some \(B_q\), together with \(|4\rangle\). Let \(D\) be a classical \(\text{QLS}(7)\). Thus Lemma~\ref{lem:3.5} applies with \(c_1=5\), \(c_2=4\), and \(c_3=7\), and all cardinalities \(c\in [4u+7,\,16u(u-7)+7]\setminus\{4u+8\}\) are attainable by \(\text{QLS}(4u+7)\text{s}\).

\textbf{Case 2.}
By Lemma~\ref{lem:4.3}, for each \(c_1\in\{44,87,130\}\), we may choose \(C_0,\ldots,C_6\) such that \(|4\rangle\in C_r\) for  \(r=0,\ldots,6\), \(\left|\bigcup_{r=0}^{6}C_r\right|=c_1\), and \(\left(\bigcup_{r=0}^{6}C_r\right)\cap \left(\bigcup_{q=0}^{u-8}B_q^+\right)=\emptyset\). Let \(D\) be a classical \(\text{QLS}(7)\). Thus Lemma~\ref{lem:3.5} applies with \(c_1=44\), \(c_2=0\), and \(c_3=7\), and all cardinalities \(c\in [47u+7,\,16u^2-69u+7]\setminus\{47u+8\}\) are attainable by \(\text{QLS}(4u+7)\text{s}\).

Similarly, \(c_1=87\) and \(c_1=130\) yield the attainable cardinalities \(c\in [90u+7,\,16u^2-26u+7]\setminus\{90u+8\}\) and \(c\in [133u+7,\,16u^2+17u+7]\setminus\{133u+8\}\), respectively.

\textbf{Case 3.}
By Lemma~\ref{lem:4.4}, we may choose \(C_0,\ldots,C_6\) such that \(|4\rangle\notin C_r\) for \(0\le r\le 6\), \(\left|\bigcup_{r=0}^{6}C_r\right|=168\), and \(\left(\bigcup_{r=0}^{6}C_r\right)\cap \left(\bigcup_{q=0}^{u-8}B_q^+\right)=\emptyset\). Let \(D\) be a \(\text{QLS}(7)\) with \(|D|=49\). Thus Lemma~\ref{lem:3.6} applies with \(c_1=168\), \(c_2=0\), and \(c_3=49\), and all cardinalities \(c\in [172u+49,\,(4u+7)^2]\setminus\{172u+50\}\) are attainable by \(\text{QLS}(4u+7)\text{s}\).

Therefore, for \(u\ge 10\), all cardinalities \(c\in [4u+7,\,(4u+7)^2]\setminus\{4u+8\}\) are attainable by \(\text{QLS}(4u+7)\text{s}\). The cases \(u\in [2,9]\setminus\{4\}\) follow from Table~\ref{tab:4u+7}.
\end{proof}

\begin{table}[htbp]
\centering
{\footnotesize
\setlength{\tabcolsep}{3pt}
\caption{Cardinalities of $\text{QLS}(4u+7)$s for $u\in [2,9]\setminus\{4\}$}
\label{tab:4u+7}
\begin{tabular}{
c c
@{\hspace{5pt}}
p{3.5cm}
@{\hspace{3pt}}
p{6.0cm}
@{\hspace{38pt}}   
c
}

\toprule
$u$ & $4u+7$ & Constructions & Input designs & Cardinalities \\
\midrule

\multirow{2}{*}{$2$} & \multirow{2}{*}{$15$}
&
\begin{tabular}[t]{@{}l@{}}
Construction~\ref{con:2}\\
$m=3,n=5$
\end{tabular}
&
\begin{tabular}[t]{@{}l@{}}
Lemma~\ref{lem:4.20},  three  \(\text{QLS}(5)\text{s}\)  with\\ $c\in [5,75]\setminus\{6\}$.
\end{tabular}
&
$[15,225]\setminus\{16\}$ \\
\midrule

\multirow{2}{*}{$3$} & \multirow{2}{*}{$19$}
&
\begin{tabular}[t]{@{}l@{}}
Lemma~\ref{lem:3.5},\\
$n=3,m=6,h=1$
\end{tabular}
&
\begin{tabular}[t]{@{}l@{}}
Example~\ref{ex:1}, $\text{QLS}(6)\text{s}$ admit cardinalities\\ $\{6,8,9\}$; Corollaries~\ref{cor:2.7},   \ref{cor:2.11}, \ref{cor:2.12}, two  \(\text{QLS}(6)\text{s}\)\\ with  $c_0\in [6,18]\setminus\{7\}$; $c_1=7$,
$ c_2=6, c_3=1$.
\end{tabular}
&
$[19,55]\setminus\{20\}$ \\

\cmidrule(lr){3-5}

& &
\begin{tabular}[t]{@{}l@{}}
Lemma~\ref{lem:3.5},\\
$n=4,m=4,h=3$
\end{tabular}
&
\begin{tabular}[t]{@{}l@{}}
$c_0\in  \{4,6,8,16\}$;
Lemma~\ref{lem:4.3}, $c_1\in \{9,11,12\}$,\\$c_2=0; c_3=3$.

\end{tabular}
&
$ [51,111]\setminus\{52\}$
 \\

\cmidrule(lr){3-5}

& &
\begin{tabular}[t]{@{}l@{}}
Lemma~\ref{lem:3.6},\\
$n=4,m=4,h=3$
\end{tabular}
&
\begin{tabular}[t]{@{}l@{}}
$c_0\in  \{4,6,8,16\}$;
Lemma~\ref{lem:4.4}, $c_1 \in T_3,\, c_2=0$;\\$c_3=3$.
\end{tabular}
&
$[99,355]\setminus\{100\}$ \\

\cmidrule(lr){3-5}

& &
\begin{tabular}[t]{@{}l@{}}
Lemma~\ref{lem:3.6},\\
$n=3,m=6,h=1$
\end{tabular}
&
\begin{tabular}[t]{@{}l@{}}
\(\text{QLS}(6)\text{s}\) admit  cardinality \(36\);
Corollaries~\ref{cor:2.7},\\   \ref{cor:2.11},  \ref{cor:2.12},  two   \(\text{QLS}(6)\text{s}\) with $c_0\in [70,72]$;\\
Lemma~\ref{lem:4.21}, $c_1 =48, c_2=0$; $c_3=1$.
\end{tabular}
&
$[355,361]$ \\

\midrule

\multirow{2}{*}{$5$} & \multirow{2}{*}{$27$}
&
\begin{tabular}[t]{@{}l@{}}
Lemma~\ref{lem:3.5},\\
$n=5,m=5,h=2$
\end{tabular}
&
\begin{tabular}[t]{@{}l@{}}
 Lemma~\ref{lem:4.20},  $c_0\in [5,75]\setminus\{6\}$;\\
$c_1=6,  c_2=5$; $c_3=2$.
\end{tabular}
&
$[27,377]\setminus\{28\}$ \\

\cmidrule(lr){3-5}

& &
\begin{tabular}[t]{@{}l@{}}
Lemma~\ref{lem:3.6},\\
$n=5,m=5,h=2$
\end{tabular}
&
\begin{tabular}[t]{@{}l@{}}
 Lemma~\ref{lem:4.20}, $c_0\in [5,75]\setminus\{6\}$;\\
Lemma~\ref{lem:4.21}, $c_1 \in \{35,70\}, c_2=0$; $c_3=2$.
\end{tabular}
&
$[202,727]\setminus\{203\}$ \\

\cmidrule(lr){3-5}

& &
\begin{tabular}[t]{@{}l@{}}
Construction~\ref{con:2},\\
$m=3,n=9$
\end{tabular}
&
\begin{tabular}[t]{@{}l@{}}
\(\text{QLS}(9)\text{s}\) admit  cardinality \(81\);
Corollaries~\ref{cor:2.7},\\   \ref{cor:2.11},  \ref{cor:2.12},  three  \(\text{QLS}(9)\text{s}\) with    $c\in [241,243]$.
\end{tabular}
&
$[723,729]$ \\
\midrule

\multirow{2}{*}{$6$} & \multirow{2}{*}{$31$}
&
\begin{tabular}[t]{@{}l@{}}
Lemma~\ref{lem:3.5},\\
$n=3,m=10,h=1$
\end{tabular}
&
\begin{tabular}[t]{@{}l@{}}
Lemma~\ref{lem:4.6}, $\text{QLS}(10)\text{s}$ admit cardinalities\\ $[10,200]\setminus\{11\}$; Corollaries~\ref{cor:2.7}, two  \(\text{QLS}(10)\text{s}\)\\ with   $c_0\in [10,200]\setminus\{11\}$;
$c_1=11, c_2=10$, $c_3=1$.
\end{tabular}
&
$[31,601]\setminus\{32\}$ \\

\cmidrule(lr){3-5}

& &
\begin{tabular}[t]{@{}l@{}}
Lemma~\ref{lem:3.6},\\
$n=3,m=10,h=1$
\end{tabular}
&
\begin{tabular}[t]{@{}l@{}}
 Lemma~\ref{lem:4.6} and Corollary~\ref{cor:2.7}, \\
 $c_0\in [10,200]\setminus\{11\}$;\\
Lemma~\ref{lem:4.21}, $c_1=120, c_2=0$; $c_3=1$.
\end{tabular}
&
$[391,961]\setminus\{392\}$ \\

\midrule

\multirow{2}{*}{$7$} & \multirow{2}{*}{$35$}
&
\begin{tabular}[t]{@{}l@{}}
Construction~\ref{con:2}\\
$m=7,n=5$
\end{tabular}
&
\begin{tabular}[t]{@{}l@{}}
 Lemma~\ref{lem:4.20},  seven  \(\text{QLS}(5)\text{s}\)  with\\ $c\in [5,175]\setminus\{6\}$.
\end{tabular}
&
$[35,1225]\setminus\{16\}$ \\
\midrule

\multirow{2}{*}{$8$} & \multirow{2}{*}{$39$}
&
\begin{tabular}[t]{@{}l@{}}
Construction~\ref{con:2}\\
$m=3,n=13$
\end{tabular}
&
\begin{tabular}[t]{@{}l@{}}
Lemma~\ref{lem:4.5}, $\text{QLS}(13)\text{s}$ admit cardinalities\\ $[13,169] \setminus\{14\}$; Corollaries~\ref{cor:2.7}, three  \(\text{QLS}(13)\text{s}\)\\ with  $c\in [13,507]\setminus\{14\}$.
\end{tabular}
&
$[39,1521]\setminus\{40\}$ \\
\midrule

\multirow{2}{*}{$9$} & \multirow{2}{*}{$43$}
&
\begin{tabular}[t]{@{}l@{}}
Lemma~\ref{lem:3.5},\\
$n=3,m=14,h=1$
\end{tabular}
&
\begin{tabular}[t]{@{}l@{}}
Lemma~\ref{lem:4.6}, $\text{QLS}(14)\text{s}$ admit cardinalities\\ $[14,196]\setminus\{15\}$; Corollaries~\ref{cor:2.7}, two  \(\text{QLS}(14)\text{s}\)\\ with   $c_0\in [14,392]\setminus\{15 \}$;
$c_1=15,c_2=14$, $c_3=1$.
\end{tabular}
&
$[43,1177]\setminus\{44\}$ \\

\cmidrule(lr){3-5}

& &
\begin{tabular}[t]{@{}l@{}}
Lemma~\ref{lem:3.6},\\
$n=3,m=14,h=1$
\end{tabular}
&
\begin{tabular}[t]{@{}l@{}}
Lemma~\ref{lem:4.6} and   Corollary~\ref{cor:2.7},\\  $c_0\in [14,392]\setminus\{15 \}$;\\
Lemma~\ref{lem:4.21}, $c_1=224, c_2=0$; $c_3=1$.
\end{tabular}
&
$[715,1849]\setminus\{716\}$ \\

\bottomrule
\end{tabular}
} 
\end{table}

\noindent\textit{\textbf{Proof of Theorem~\ref{th:1.2}.}}
By   Lemmas~\ref{lem:4.2}, \ref{lem:4.5}--\ref{lem:4.7}, for integers \( v \ge 8 \), \( v \ne 9,11 ,23\), and \( c \in [v, v^{2}] \setminus \{v+1\}\),
there exists a \(\text{QLS}(v)\) with cardinality \( c \).

\section{Concluding remarks}
In this paper, we investigate the possible cardinalities of the $\text{QLS}(v)$. By extending the classical singular direct product construction and direct product construction to the quantum setting, we show that for any integer \(v \ge 8\) with \(v \ne 9,11,23\), a $\text{QLS}(v)$ exists for every cardinality \(c \in [v, v^{2}] \setminus \{ v+1 \}\).

Quantum Latin squares can also be constructed by several other methods, including those based on \(m \times n\) row-quantum Latin rectangles \cite{Zhang1} and complex Hadamard matrices  \cite{Mu2}, in addition to the singular direct product and direct product  constructions presented in this paper. Since genuinely quantum Latin squares of order $2$ or $3$ cannot exist \cite{Paczos}, the possible cardinalities of quantum Latin squares of order $4$ are $4$, $6$, $8$, and $16$ \cite{Zhang1}, and those of order $5$ are $5$, $7$, $12$, $21$, $24$, and $25$ \cite{Zhang4}. Thus, the problem of possible cardinalities for quantum Latin squares remains open only for orders $v=6,7,9,11$, and $23$. Table~\ref{tab:5} lists the attainable and uncertain cardinalities of $\text{QLS}(v)$ for these values of $v$.

\begin{table}[H]
\scriptsize
\centering
\caption{Attainable and uncertain cardinalities of a \(\text{QLS}(v)\) for \(v=6,7,9,11,23\)}
\label{tab:5}

\setlength{\tabcolsep}{3pt}

\begin{tabular}{|c|p{7.7cm}|p{6.3cm}|}
\hline
\(v\) & Attainable cardinalities & Uncertain cardinalities \\
\hline

\(6\) &
\([6,22]\setminus\{7\}\cup\{24,26,28,30,31,33,34,36\}\) &
\(\{23,25,27,29,32,35\}\) \\
\hline

\(7\) &
\([7,23]\setminus\{8\}\cup\{28,43,44,46,47,48,49\}\) &
\([24,27]\cup[29,42]\cup\{45\}\) \\
\hline

\(9\) &
\([9,59]\setminus\{10,34\}\cup\{63,64,72,75,76,77,78,79,80,81\}\) &
\([60,62]\cup[65,71]\cup\{34,73,74\}\) \\
\hline

\(11\) &
$[20,53]\setminus\{21,48,50\}\cup\{11,105,111,112\} \cup[116,121]\setminus\{119\}$ &
$[13,19]\cup[54,115]\cup\{ 21,48,50,105,111,112,119\}$ \\
\hline

\(23\) &
\([23,529]\setminus\{24,528\}\) &
\(\{528\}\) \\
\hline
\end{tabular}
\end{table}

\noindent{\bf Acknowledgments} The authors would like to acknowledge Prof. Lie Zhu (Suzhou University) for his helpful discussions during the preparation of this manuscript.  H. Cao's research was supported by the National Natural Science Foundation of China (Grants No. 12471313 and No. 12071226).

\end{document}